\documentclass[twocolumn]{autart}    

\usepackage{graphicx}          
\usepackage{amsmath}
\usepackage{amssymb}
\usepackage{subfig}
\usepackage{bbm}
\usepackage{cite}
\usepackage{enumerate}
\usepackage{algorithm}
\usepackage{color}
\usepackage[noend]{algpseudocode}
\newtheorem{assumption}{Assumption}
\newtheorem{definition}{Definition}

\newtheorem{proposition}{Proposition}
\newtheorem{remark}{Remark}
\newtheorem{lemma}{Lemma}
\newtheorem{corollary}{Corollary}
\newcommand{\RN}[1]{%
	\textup{\uppercase\expandafter{\romannumeral#1}}%
}
\makeatletter
\def\BState{\State\hskip-\ALG@thistlm}
\makeatother
\begin{document}

\begin{frontmatter}

\title{Online Convex Optimization Using Coordinate Descent Algorithms\thanksref{footnoteinfo}} 

\thanks[footnoteinfo]{This paper was not presented at any conference. This work was supported by the Australian Research Council under the Discovery Projects DP170104099, DP210102454, and the Australian Government, via grant AUSMURIB000001 associated with ONR MURI grant N00014-19-1-2571. Corresponding author Y.~Lin.}

\author[Paestum]{Yankai Lin}\ead{y.lin2@tue.nl},    
\author[Rome]{Iman Shames}\ead{iman.shames@anu.edu.au},               
\author[Baiae]{Dragan Nesic}\ead{dnesic@unimelb.edu.au}  
\address[Paestum]{Department of Mechanical Engineering, Eindhoven University of Technology, the Netherlands}  
\address[Rome]{CIICADA Lab, School of Engineering, The Australian National University, Acton, 0200, ACT, Australia}             
\address[Baiae]{Department of Electrical and Electronic Engineering, The University of Melbourne, Parkville, 3010, Victoria, Australia}   
          

\begin{abstract}                          
This paper considers the problem of online optimization where the objective function is time-varying. In particular, we extend coordinate descent type algorithms to the online case\color{black}, where the objective function varies after a finite number of iterations of the algorithm. Instead of solving the problem exactly at each time step, we only apply a finite number of iterations at each time step\color{black}. Commonly used notions of regret are used to measure the performance of the online algorithm. Moreover, coordinate descent algorithms with different updating rules are considered, including both deterministic and stochastic rules that are developed in the literature of classical offline optimization. A thorough regret analysis \color{black} is \color{black} given for each case. Finally, numerical simulations are provided to
illustrate the theoretical results.
\end{abstract}

\end{frontmatter}

\section{Introduction}
\color{black}Online learning \cite{Shwartz12Found},  resource allocation \cite{chen2017online}, demand response in power systems \cite{LT18}, and localization of moving targets \cite{BSR18} are just a few examples where online convex optimization (OCO) has been applied\color{black}. In the problem setup of OCO, the objective functions are time-varying and are not available to the decision maker \emph{a priori}. At each time instance, after an update of the decision variable, new information of the latest cost function is made available to the decision maker. The objective of the decision maker is to minimize the objective function over time. One commonly used performance measure of an online optimization algorithm is the notion of regret which measures the \color{black} suboptimality \color{black} of the algorithm \color{black} compared \color{black} to the outcome generated by the best \color{black} decision at each time step.\color{black} 

In the seminal work of \cite{zinkevich2003online}, the method of online gradient descent is proposed for OCO problems, where at each time step the decision maker performs one gradient descent step using the latest available information. A static regret upper bound that is sublinear in $T$ is proved, where $T$ is the length of the horizon. Under stronger assumptions on the cost functions such as strong convexity, an improved logarithmic regret bound can be achieved \cite{hazan2007logarithmic,mokhtari2016online,garber2016linearly}. If future information is available, it can be used to further improve the performance of the online optimization algorithm in terms of regret bounds. The work \cite{lesage2020predictive} introduces an additional predictive step following the algorithm developed in \cite{zinkevich2003online}, if certain conditions on the estimated gradient and descent direction are met. Similar algorithms have also been extended to cases where zeroth order \cite{cao2019online,shames2019online,yi2020distributed,pang2023randomized,xiong2022distributed} and second order \cite{lesage2020second} oracles are used instead of (sub)gradients\color{black}. The works \cite{cao2019online,yi2020distributed} on bandit feedback consider the situation where there are time-varying inequality constraints\color{black}. In such cases the algorithms proposed in \cite{zinkevich2003online} will be hard to implement because of the high computational resource demand of the projection operation. This motivates recent research on online optimization algorithms with time-varying constraints including the primal-dual algorithm proposed in \cite{jenatton2016adaptive,yi2023regret}, a modified saddle-point method given in \cite{chen2017online}. Other algorithms are also proposed to handle stochastic constraints \cite{yu2017online} and cover continuous-time applications \cite{paternain2017online}. In the case where only the values rather than the exact form of the cost function at are revealed to the decision maker, bandit feedback based online algorithms \cite{cao2019online,yi2020distributed} can be used to solve the problem, other methods such as forward gradient \cite{mafakheri2022first} have also been proposed recently to deal with the issue\color{black}. The need for applications in large-scale systems has also led to extensive research on distributed OCO. Distributed online algorithms that achieve sublinear regret bound for convex optimization problems with static constraints can be found in \cite{shahrampour2018distributed,hosseini2016online,yan2013distributed}. For instance, \cite{shahrampour2018distributed} proposes a distributed version of the dynamic mirror descent algorithm which is a generalization of the classical gradient descent methods suitable for high-dimensional optimization problems. The work \cite{li2018distributed} proposes distributed online primal{-}dual algorithms for optimization problems with static coupled inequality constraints while the work \cite{yi2020distributedsp} studies distributed online convex optimization with time-varying inequality constraints in the discrete-time setting\color{black}. For a more detailed documentation of recent advances of online optimization, we refer the readers to the survey paper \cite{li2022survey}. 

\color{black}To the best of our knowledge, Coordinate descent \cite{Wright}, as an important class of optimization algorithms, is not sufficiently analyzed by researchers in the online optimization community.  In coordinate descent algorithms, most components of the decision variable are fixed during one iteration while the cost function is minimized with respect to the remaining components of the decision variable. The resulting problem is lower-dimensional and often much easier to solve. Thus, coordinate descent algorithms have great potential in applications such as machine learning, where iteration with full gradient information is computationally expensive. In \cite{nesterov2012efficiency}, it is shown that for huge scale problems, coordinate descent can be very efficient. Another situation where one may find coordinate descent useful is dual decomposition based methods for distributed optimization, see \cite{lin2021asynchronous} and references therein. Specifically, the dual problem of multi-agent optimal consensus results in a sum of functions with very loose coupling between the dual variables. Calculation of a component of the gradient of the dual function only involves computations and communications of a pair of agents (or processors). Moreover, it can also be implemented in a parallel fashion as shown in \cite{BT}. Therefore, sufficient effort has been made recently by researchers to develop theoretical performance guarantees of various coordinate descent algorithms \cite{Wright}. In this paper, we aim to extend this appealing algorithm to solve OCO problems by providing an in-depth regret analysis for different types of online coordinate descent algorithms. \color{black}

The main contributions of the paper can be summarized as follows. First, we extend the coordinate descent algorithms considered in \cite{Wright} to the online case and provide their regret analysis. To the best of our knowledge, this is the first attempt to look at possibilities of using coordinate descent methods to solve OCO problems. Second, we provide an in-depth regret analysis of various coordinate descent algorithms with different rules, such as cyclic updating rules and random updating rules. Specifically, we consider both random and deterministic online coordinate descent algorithms under assumptions commonly used in the literature\color{black}. In particular, most existing literature on OCO are based on extensions of offline algorithms that monotonically reduce the distance from the decision variable to the set of solutions at each iteration. An example is the well-known online gradient descent \cite{zinkevich2003online,hazan2007logarithmic}. However, offline deterministic coordinate descent algorithm, although has provable convergence properties to the set of solutions, does not necessarily result in an updated variable that is closer to the set of solutions at each iteration. We overcome this issue by using predictive like updates at each time which are detailed in Section \ref{RB-D}\color{black}. Lastly, we show that the regret bounds achieved by our online coordinate descent algorithms are comparable to those achieved by the literature on centralized full-gradient based online algorithms.\color{black}

\color{black}We summarize the theoretical upper bounds of regrets we prove in Theorem \ref{T1-Ch4} to Theorem \ref{T8-Ch4} for Algorithms \ref{alg1-Ch4}, \ref{alg4-Ch4}, \ref{alg5-Ch4} in the following table.
\begin{table}[h!]
	\caption{Regret bounds proved in the paper with $C_T=\sum_{t=1}^T|x_t^*-x_{t-1}^*|,C_{T,2}=\sum_{t=1}^T|x_t^*-x_{t-1}^*|^2$ and $C$ stands for convex cases $SC$ stands for strongly convex cases.}
	\centering
	\begin{tabular}{ |p{1.2cm}||p{1cm}||p{1.1cm}||p{1.3cm}||p{1cm}|}
		\hline
		\  &$R_T^s$,C & $R_T^s$,SC  & $R_T^d$,C& $R_T^d$,SC\\ 
		\hline\hline
		Alg. \ref{alg1-Ch4} & $O(\sqrt{T})$ & $O(\log T)$ & $O(\sqrt{C_TT})$ & $O(C_T)$\\ 
		\hline
		Alg. \ref{alg2-Ch4},\ref{alg4-Ch4} & $O(\sqrt{T})$ & $O(\log T)$ & $O(\sqrt{C_TT})$ & $O(C_{T,2})$\\ 
		\hline
		Alg. \ref{alg3-Ch4},\ref{alg5-Ch4} & $O(\sqrt{T})$ & $O(\log T)$ & $O(\sqrt{C_TT})$ & $O(C_{T,2})$\\ 
		\hline
		\color{black}Gradient Method\color{black} & $O(\sqrt{T})$\ \cite{zinkevich2003online} & $O(\log T)$\ \cite{hazan2007logarithmic} & $O(\sqrt{C_TT})$\ \cite{cao2021online}& $O(C_T)$\ \cite{shahrampour2018distributed}\ $O(C_{T,2})$\ \cite{chang2021online}\\ 
		\hline
	\end{tabular}
	\label{table1-Ch4}
\end{table}

The regret bounds summarized in Table~\ref{table1-Ch4} are consistent with regret bounds of full-gradient based online optimization algorithms proved in the existing literature \cite{Shwartz12Found,hazan2007logarithmic,cao2021online,mokhtari2016online} under similar settings. Our dynamic regret bounds for strongly convex functions proved in Theorems \ref{T7-Ch4} and \ref{T8-Ch4} might need multiple updates at each time $t$. This setup is also adopted in some existing works including \cite{zhang2017improved,chang2021online} to achieve less conservative regret bounds. It should also be noted that, although the online algorithms with different update rules share the regret bounds with the same order, the exact coefficient for the regret bounds may still be different.

\color{black}The rest of paper is organized as follows. The problem formulation is presented in Section~\ref{PF4}. The online coordinate descent algorithm considered in this paper is given in Section~\ref{OCD}. Regret bounds for random online coordinate descent algorithms are given in Section~\ref{RB-R} followed by regret bounds for deterministic online coordinate descent algorithms in Section~\ref{RB-D}. The numerical simulation is given in Section~\ref{NS4}. Finally the results presented in this paper is summarized in Section~\ref{Summary4}.

\textit{Notation}: Let $\mathbb{R}$ be the set of real numbers and $\mathbb{R}^{n}$ be the $n$-dimensional Euclidean space, $\mathbb{R}_{\geq 0}$ (resp. $\mathbb{R}_{>0}$) be the set of non-negative (resp. positive) real numbers\color{black}. For $x,y\in\mathbb{R}^n$, $\langle x,y\rangle$ denotes the inner product in $\mathbb{R}^n$\color{black}. The set of non-negative (resp. positive) integers is denoted by $\mathbb{Z}_{\geq 0}$ (resp. $\mathbb{Z}_{>0}$). Furthermore, $|\cdot|$ denotes the Euclidean norm of a vector $x\in \mathbb{R}^{n}$. The matrix $I_n$ is used to denote the $n$-dimensional identity matrix and $n$ will be omitted when the dimension is clear. For a given vector $x$, $x_{(i)}$ \color{black} denotes the $i$-th component of $x$, $x_i$ denotes the value of $x$ at time $i$, and $x_{(i),j}$ denotes the value of the $i$-th component of $x$ at time $j$. All random variables considered in this paper are defined on a probability space $(\Omega,\mathcal{F},\mathbb{P})$, where $\Omega$ is the sample space, the $\sigma$-field $\mathcal{F}$ is the set of events and $\mathbb{P}$ is the probability measure defined on $(\Omega,\mathcal{F})$\color{black}. 

\section{Problem formulation}
\label{PF4}
We consider the following optimization problem\color{black}
\begin{equation}\label{MainProblem}
	\underset{x\in\Theta}{\text{min}}\ f_t(x),
\end{equation}\color{black}
where $f_t:\mathbb{R}^n\rightarrow \mathbb{R}$ is the convex cost function at time $t$, $x\in\mathbb{R}^n$ is the decision variable\color{black}, and $\Theta\subseteq \mathbb{R}^n$ is the non-empty closed convex constraint set\color{black}. For simplicity, we assume $f_t$ is continuously differentiable for any $t\in\mathbb{Z}_{\geq 0}$. Moreover, let the decision variable $x_t$ at any given time $t$ be partitioned as $x_t=[x_{(1),t}^T,x_{(2),t}^T,\ldots,x_{(P),t}^T]$, $x_{(p),t}\in\mathbb{R}^{n_p}$, $\sum_{p=1}^{P}n_p=n$. The $P$ components of the vector \color{black} are \color{black} assigned to $P$ individual processors. For any integer \color{black} $p\in \{1,\ldots P\}$\color{black}, the processor $p$ is responsible for updating its own ``local'' decision variable \color{black} $x_{(p),t}$ \cite{BT}\color{black}. 

\color{black}Denote the minimizer of $f_t(x)$ at time $t$ by $x_t^*\in\Theta$\color{black}. We use a notion of regret to measure processors' ability to track $x_t^*$\color{black}. Regret refers to a quantity that measures the overall difference between the cost incurred by an algorithm and the cost at the best possible point from an offline view up to a horizon $T$\color{black}. Two notions of regret commonly considered in the literature are static regret and dynamic regret. The static regret is defined as:
\begin{equation}\label{S-static-R}\color{black}
\begin{split}
R_T^s:&=\sum_{t=1}^{T}f_t(x_t)-\underset{x}{\text{min}}\sum_{t=1}^{T}f_t(x)\\
      &=\sum_{t=1}^{T}f_t(x_t)-\sum_{t=1}^{T}f_t(x^*),
\end{split}
\end{equation}
where \color{black} $x^*:=\arg\underset{x\in\Theta}\min\sum_{t=1}^{T}f_t(x)$ for a given $T$\color{black}, the subscript $T$ is omitted for convenience\color{black}. On the other hand, the dynamic regret is defined as
\begin{equation}\label{S-dynamic-R}
	R_T^d:=\sum_{t=1}^{T}f_t(x_t)-\sum_{t=1}^{T}f_t(x_t^*).
\end{equation}
\color{black}
\begin{remark}
	\label{Regret-rem}
	Static regret is a useful performance metric in applications such as static parameter estimation, where the variable of interest is static. However, if the variable of interest evolves over time (e.g. tracking moving targets), the notion of dynamic regret makes more sense than its static counterpart. It can be seen from (\ref{S-static-R}) and (\ref{S-dynamic-R}) that regret captures the accumulation of errors due to the fact that optimization problems are not solved exactly at each step. If the regret is sublinear in $T$, then the average accumulated error $R_T^s/T$ (or $R_T^d/T$) will converge to $0$ as $T\rightarrow\infty$. This further implies that $x_t$ converges to $x^*$ (or $x_t^*$). Although being the more appropriate performance metric in most applications, dynamic regret does not have a provable bound sublinear in $T$ in general. To obtain a sublinear regret bound, additional regularity assumptions such as bounded variation of environment are typically required, see \cite{shahrampour2018distributed} for example.
\end{remark}
\color{black}If stochastic algorithms are considered, similar notions of regrets can be defined via expectations. 
\begin{equation}\label{SRR}
	R_T^s:=\sum_{t=1}^{T}\mathbb{E}[f_t(x_t)]-\underset{x}{\text{min}}\sum_{t=1}^{T}f_t(x).
\end{equation}
\begin{equation}\label{DRR}
	R_T^d:=\sum_{t=1}^{T}\mathbb{E}[f_t(x_t)]-\sum_{t=1}^{T}f_t(x_t^*).
\end{equation}
With some abuse of notation, $R_T^s$ and $R_T^d$ are used for both random and deterministic cases. However, this would not lead to confusion since it should be clear whether an algorithm is \color{black} stochastic \color{black} or not. 

\section{Online coordinate descent algorithms}\label{OCD}
We construct our online block coordinate descent algorithms following the setup in \cite{Wright}. At time $t$, we select a component of $x_t$ to update. The updating component at time $t$ is denoted by $i_t$\color{black}. If the $i$-th component is updating, then updating equation for component $i$ is given by
\begin{equation}
	x_{(i),t+1}=x_{(i),t}-\alpha_t[\nabla_{(i)} f_{t}(x_{t})],
\end{equation}
\color{black}where $\alpha_t$ is the stepsize at time $t$ and \color{black} $\nabla_{(i)} f_{t}(x_{t})$ \color{black} is the $i$-th component of the gradient of $f$ evaluated at $x_t$ at time $t$. For any $p\neq i_t$, we have
\begin{equation}
	x_{(p),t+1}=x_{(p),t}.
\end{equation}
\color{black}Then $x_{t+1}$ is projected onto $\Theta$\color{black}. We define the matrix $U_t=\text{diag}(U_{(1),t},U_{(2),t},\ldots,U_{(P),t})$. For any $t\in\mathbb{Z}_{\geq 0}$ and $1\leq p\leq P$, $U_{(p),t}\in\{I_{n_p},0_{n_p}\}$, where $I_{n_p}$ and $0_{n_p}$ denote identity and zero matrices of dimension $n_p$, respectively\color{black}. Then the updates of the coordinate descent algorithm at time $t$ can be written as 
\begin{equation}\label{coord}
	x_{t+1}=\Pi_{\Theta}(x_{t}-\alpha_tU_t[\nabla f_t(x_{t})]),
\end{equation}
where $\Pi_{\Theta}(\cdot)$ denotes projection on $\Theta$ which is well defined by closedness and convexity of $\Theta$. The following non-expansive property of the projection operator will be used extensively throughout the paper.
\begin{lemma}\cite[Proposition 3.2.1]{Bert1}
	\label{Proj_nonexp}
	Let $\Theta$ be a non-empty closed convex set. We have $|\Pi_{\Theta}(x)-\Pi_{\Theta}(y)|\leq|x-y|$, for any $x,y\in\mathbb{R}^n$.
\end{lemma}

\color{black}In this paper, we consider the following three commonly used schemes of selecting the updating component $i_t$.
\begin{itemize}
	\item The random coordinate descent algorithm. In this case, $i_t$ is selected randomly with equal probability, independently of the selections made at previous iterations.
	\item The cyclic coordinate descent algorithm. In this case, the updating component $i_t$ is selected in a pre-determined cyclic fashion: $i_{t+1}=(i_t\ \mathrm{mod}\ {P})+1$.
	\item The coordinate descent algorithm with Gauss-Southwell Rule. In this case, the updating component $i_t$ is selected as $i_t\in\arg \max_i|\nabla_{(i)} f_t(x_t)|$.
\end{itemize} 
\color{black}
In the paper, we consider the case where only one coordinate is allowed to change per iteration. The results on random coordinate descent can be extended to other cases where probabilities of selections are unequal with potentially overlapping components \cite{LSN}, such as the random sleep scheme \cite{yi2015stochastic}. Intuitively speaking, if for any $t$ the expectation of the update direction takes the form $\Gamma\nabla f_t(x)$ for a given positive definite diagonal matrix $\Gamma$, then the analysis in this work can be applied with mild modifications. Our analysis for deterministic cases can not be trivially extended to cover overlapping components. We aim to address this topic in future research.
\begin{remark}
	\label{Remark-compare}
	 A substantial review of variants of coordinate descent algorithms can be found in \cite[Section 6.5.1]{Bert1}. The cyclic selection of coordinates is normally assumed to ensure convergence of the algorithm. On the other hand, the use of an irregular order is then considered by researchers to accelerate convergence. Particularly, it is shown in \cite{Wright} that randomization leads to faster convergence in terms of expectation\color{black}. Obviously, this is not guaranteed for each instance of the algorithm. The Gauss-Southwell method leads to \color{black} faster convergence at the cost of extra computations and evaluations of gradients during the selection of coordinates which can be an issue in large-scale problems \cite{nutini2015coordinate}.
\end{remark}
\color{black}

\section{\color{black}Regret bounds for online coordinate descent algorithms with random coordinate selection rules}\label{RB-R}
The online random coordinate descent algorithm considered in this section is given in Algorithm \ref{alg1-Ch4}.
\begin{algorithm}
	\caption{Online random coordinate descent algorithm}\label{alg1-Ch4}
	\begin{algorithmic}[1]
		\State \textbf{Initialization}:  $x_0\in\Theta$.
		\State \textbf{Coordinate Selection}: At time $t$, $i_t=p$ with probability $\frac{1}{P}$ where $p\in\{1,2,\ldots,P\}$.
		\State \textbf{Update}:  For $i=i_t$:
		\begin{align*}
			x_{(i),t+1}\leftarrow x_{(i),t}-\alpha_t[\nabla_{(i)} f_{t}(x_{t})].
		\end{align*}
		For all $i\neq i_t$:
		\begin{equation*}
			x_{(i),t+1}\leftarrow x_{(i),t}.
		\end{equation*}
		\State \textbf{Projection}:  \color{black}For $x_{t+1}\leftarrow\Pi_{\Theta}(x_{t+1})$.\color{black}	
		\State Set $t\leftarrow t+1$ and go to Step 2.
	\end{algorithmic}
\end{algorithm}
\subsection{Static regret for convex functions}
Before we state the main results, we first list the assumptions.
\begin{assumption}\label{a1-Ch4}
	For any given $t\in\mathbb{Z}_{>0}$, $x\in\Theta$ and \color{black} $x^*:=\arg\underset{x\in\Theta}\min\sum_{t=1}^{T}f_t(x)$\color{black}, the following inequalities hold uniformly in $t$ \color{black} and $T$ \color{black} for problem (\ref{MainProblem}),
	\begin{enumerate}[(i)]
		\item $|\nabla f_t(x)|\leq G$,
		\item $|x_t-x^*|\leq R$,  
	\end{enumerate} 
\color{black}where $x_t$ is the decision variable at time $t$.
\end{assumption}
\color{black}
\begin{remark}\label{CH4-R2}
	Item (ii) of Assumption \ref{a1-Ch4} can be ensured if the constraint set $\Theta$ is bounded. Moreover, when $\Theta$ is bounded, item (i) of Assumption \ref{a1-Ch4} holds in many cases including linear regression and logistic regression \cite{cao2022decentralized}. By \cite[Lemma 2.6]{Shwartz12Found}, Assumption \ref{a1-Ch4} (i) implies that $f_t$ is also Lipschitz continuous uniformly in $t$ over $\Theta$. These assumptions are quite standard in the literature on online optimization \cite{cao2019online,cao2021online,cao2022decentralized,li2022survey,yi2023regret,yuan2022distributed,lesage2020predictive,li2021online,shahrampour2018distributed}. 
\end{remark}

\color{black}Before we state the first main result of the paper, we first present the so-called doubling trick scheme stepsize rule which is introduced in \color{black}\cite[Section 2.3.1]{Shwartz12Found}.\color{black} 
\begin{definition}[Doubling trick scheme]\label{DTS}
	A stepsize sequence $\{\alpha_t\}$ is said to be chosen by the doubling trick scheme, if for $q=0,1,\ldots,\lceil \log_2 T\rceil+1$, $\alpha_t=\frac{1}{\sqrt{2^q}}$ in each period of $2^q$ iterations $t=2^q, 2^q+1,\ldots,2^{q+1}-1$.
\end{definition}
\color{black}
\begin{remark}\label{doubling}
	The doubling trick scheme of stepsize choices is particularly useful when stepsizes of type $\sqrt{1/T}$ are needed to achieve a desirable regret bound. Since in applications it is often unrealistic to know the horizon $T$ in advance, by using the doubling trick scheme, regret bounds of the same order can be achieved without \color{black} explicit \color{black} knowledge of $T$. The same trick will be used extensively throughout the paper to derive regret bounds. It should also be noted that, the doubling trick scheme in general does not result a better regret bound in terms orders compared to other stepsize rules \cite{zinkevich2003online}.
\end{remark}
\color{black}
The following result states that, under Assumption \ref{a1-Ch4}, if the stepsize at each iteration is chosen by the doubling trick scheme, there is an upper bound for the static regret defined in (\ref{SRR}). Moreover, the upper bound has the order of $O(\sqrt{T})$ for convex costs.
\begin{thm}\label{T1-Ch4}
	Suppose Assumption \ref{a1-Ch4} holds. Furthermore, if the stepsize is chosen according to Definition \ref{DTS}. Then, the static regret (\ref{SRR}) achieved by Algorithm \ref{alg1-Ch4} satisfies
	\begin{equation}
		R_T^s\leq (B_1+B_2)\sqrt{T},
	\end{equation} 
	where $B_1=\frac{PR^2}{2}$ and $B_2=\frac{\sqrt{2}G^2}{2(\sqrt{2}-1)}$.
\end{thm}\color{black}
\begin{pf}
	\color{black}From (\ref{coord}) and Lemma \ref{Proj_nonexp}, we have 
	\begin{equation*}
		\begin{split}
			&|x_{t+1}-x|^2\leq|x_{t}-\alpha_tU_t[\nabla f_t(x_{t})]-x|^2=|x_t-x|^2\\
			&-2\alpha_t[U_t\nabla f_t(x_{t})]^T(x_t-x)+\alpha_t^2|U_t\nabla f_t(x_{t})|^2
		\end{split}
	\end{equation*}
	for any $x\in\Theta$\color{black}. Given $x_t$, denote the $\sigma$-field containing past data of Algorithm \ref{alg1-Ch4} up to time $t$ by $\mathcal{F}_t$\color{black}. Then, by convexity of $f_t$, we have
	\begin{equation}\label{coreinq}
		\begin{split}
			&\mathbb{E}[|x_{t+1}-x|^2|\mathcal{F}_t]\leq\\
			&|x_t-x|^2-\frac{2\alpha_t}{P}[\nabla f_t(x_{t})]^T(x_t-x)+\frac{\alpha_t^2}{P}|\nabla f_t(x_{t})|^2\\
			&\leq |x_t-x|^2-\frac{2\alpha_t}{P}(f_t(x_t)-f_t(x))+\frac{\alpha_t^2}{P}|\nabla f_t(x_{t})|^2.
		\end{split}
	\end{equation}
	Substituting \color{black} $x^*:=\arg\underset{x\in\Theta}\min\sum_{t=1}^{T}f_t(x)$ \color{black} in to (\ref{coreinq})\color{black}, taking total expectations \color{black} using  $\mathbb{E}[\mathbb{E}[x|\mathcal{F}_t]|\mathcal{F}_{t-1}] = \mathbb{E}[x|\mathcal{F}_{t-1}]$ for $\mathcal{F}_{t-1}\subset \mathcal{F}_t$\color{black}, and rearranging (\ref{coreinq}) leads to
	\begin{equation}
		\begin{split}
			&\mathbb{E}[f_t(x_t)-f_t(x^*)]\\
			&\leq \frac{P}{2\alpha_t}(\mathbb{E}[|x_t-x^*|^2]-\mathbb{E}[|x_{t+1}-x^*|^2])+\frac{\alpha_t}{2}|\nabla f_t(x_{t})|^2.
		\end{split}
	\end{equation}
	\color{black}Note that
	\begin{equation}\label{ineq2}
		\begin{split}
			R_T^s\leq&\sum_{t=1}^T\frac{P}{2\alpha_t}(\mathbb{E}[|x_t-x^*|^2]-\mathbb{E}[|x_{t+1}-x^*|^2])\\
			&+\sum_{t=1}^T\frac{\alpha_t}{2}|\nabla f_t(x_{t})|^2.
		\end{split}
	\end{equation}
	Due to Assumption \ref{a1-Ch4} the gradient of $f_t$ is uniformly bounded by $G$ and $|x_t-x^*|$ is upper bounded uniformly by $R$. Consequently\color{black},
	\begin{equation}\label{e4.13}
		\begin{split}
			&\sum_{t=1}^T\frac{P}{2\alpha_t}(\mathbb{E}[|x_t-x^*|^2]-\mathbb{E}[|x_{t+1}-x^*|^2])=\\
			&\frac{P}{2\alpha_1}\mathbb{E}[|x_1-x^*|^2]-\frac{P}{2\alpha_T}\mathbb{E}[|x_{T+1}-x^*|^2]\\
			&+\sum_{t=2}^T(\frac{P}{2\alpha_t}-\frac{P}{2\alpha_{t-1}})\mathbb{E}[|x_t-x^*|^2]\\
			&\leq \frac{PR^2}{2\alpha_1}+R^2\sum_{t=2}^T(\frac{P}{2\alpha_t}-\frac{P}{2\alpha_{t-1}})\\
			&=\frac{PR^2}{2\alpha_T}\leq \frac{PR^2}{2}\sqrt{T}:=B_1\sqrt{T}.
		\end{split}
	\end{equation}
    \color{black}In the last inequality of (\ref{e4.13}), the property $\alpha_T\geq 1/\sqrt{T}$ is used, which is a direct consequence of the definition of the stepsize rule\color{black}. On the other hand, the remaining term in (\ref{ineq2}) can be upper bounded as follows
	\begin{equation*}
		\sum_{t=1}^T\frac{\alpha_t}{2}|\nabla f_t(x_{t})|^2\leq \sum_{t=1}^T\alpha_t\frac{G^2}{2}.
	\end{equation*}
	To derive the regret bound\color{black}, we again use the properties of the doubling trick scheme\color{black}. First, we set the horizon to be some known constant $T^*$ and the stepsize is chosen as the constant $\frac{1}{\sqrt{T^*}}$. Then it is obvious that\color{black}
	\begin{equation}\label{B2}
		\sum_{t=1}^{T^{*}}\frac{\alpha_t}{2}|\nabla f_t(x_{t})|^2\leq \sum_{t=1}^{T^{*}}\alpha_t\frac{G^2}{2}\leq \frac{G^2}{2}\sqrt{T^*}.
	\end{equation}
	\color{black}Since for $q=0,1,\ldots,\lceil \log_2 T\rceil+1$, $\alpha_t=\frac{1}{\sqrt{2^q}}$ in each period of $2^q$ iterations $t=2^q, 2^q+1,\ldots,2^{q+1}-1$. That is $T^*=2^q$. Then we can sum (\ref{B2}) up over any given $T$ as
	\begin{equation}
		\begin{split}
			\sum_{t=1}^T\frac{\alpha_t}{2}|\nabla f_t(x_{t})|^2&\leq \sum_{q=0}^{\lceil \log_2 T\rceil}\sqrt{2^q}\frac{G^2}{2}\\
			&=\frac{G^2}{2}\frac{1-{\sqrt{2}}^{\lceil \log_2 T\rceil+1}}{1-\sqrt{2}}\\
			&\leq \frac{G^2}{2}\frac{1-{\sqrt{2T}}}{1-\sqrt{2}}\\
			&\leq\frac{\sqrt{2}G^2}{2(\sqrt{2}-1)}\sqrt{T}:=B_2\sqrt{T}.
		\end{split}
	\end{equation}
	Thus, the proof is complete. \hfill $\qed$
\end{pf}
\begin{remark}\label{subgradient}
	In Theorem \ref{T1-Ch4}, the differentiability of $f_t$ is in fact never used. Thus the results apply to the case where subgradients are used when $f_t$ is not differentiable for some $t$ as long as item 1 of Assumption \ref{a1-Ch4} is satisfied by the subgradients used in the iterations\color{black}. Moreover, the regret bound established in Theorem \ref{T1-Ch4} is of the same order as the one established in \cite{Shwartz12Found} for online gradient descent under the same assumptions.\color{black}  
\end{remark}

\subsection{Static regret for strongly convex functions}
In this part, we show that if the cost functions are strongly convex, then we can achieve an improved static regret bound for the online random coordinate descent algorithms. 
\begin{assumption}\label{a2-Ch4}
	For any $t\in\mathbb{Z}_{\geq 0}$, the function $f_t$ is uniformly $\mu$-strongly convex, i.e., there exists $\mu>0$ such that the following inequality holds for any $t\in\mathbb{Z}_{\geq 0}$, $x,y\in\mathbb{R}^n$
	\begin{equation}\label{StrC}
		f_t(y)\geq f_t(x)+\nabla f_t(x)^T(y-x)+\frac{\mu}{2}|y-x|^2.
	\end{equation} 	
\end{assumption}
\begin{thm}\label{T2-Ch4}
	Suppose \color{black} Assumptions \ref{a1-Ch4} (i) \color{black} and \ref{a2-Ch4} hold. Furthermore, if the stepsize is chosen as $\alpha_t=\frac{P}{\mu t}$. Then, the static regret (\ref{SRR}) achieved by Algorithm \ref{alg1-Ch4} satisfies
	\begin{equation}
		R_T^s\leq \frac{PG^2}{2\mu}(1+\log T).
	\end{equation} 
\end{thm}
\begin{pf}
	Similar to (\ref{coreinq}), given $x_t$, we have the following relationship for any $x\in\Theta$,
	\begin{equation}\label{coreinq2}
		\begin{split}
			&\mathbb{E}[|x_{t+1}-x|^2|\mathcal{F}_t]\leq\\
			&|x_t-x|^2-\frac{2\alpha_t}{P}[\nabla f_t(x_{t})]^T(x_t-x)+\frac{\alpha_t^2}{P}|\nabla f_t(x_{t})|^2\\
			&\leq |x_t-x|^2-\frac{2\alpha_t}{P}(f_t(x_t)-f_t(x))-\frac{\alpha_t\mu}{P}|x_t-x|^2\\
			&+\frac{\alpha_t^2}{P^2}|\nabla f_t(x_{t})|^2=(1-\frac{\alpha_t\mu}{P})|x_t-x|^2\\
			&-\frac{2\alpha_t}{P}(f_t(x_t)-f_t(x))+\frac{\alpha_t^2}{P}|\nabla f_t(x_{t})|^2.
		\end{split}
	\end{equation}
	\color{black}By substituting $x^*$ to (\ref{coreinq2}) and rearranging the terms, we have the following inequality\color{black}
	\begin{equation*}
		\begin{split}
			\mathbb{E}[f_t(x_t)-f_t(x^*)]\leq&\frac{P}{2\alpha_t}(\mathbb{E}[|x_t-x^*|^2]-\mathbb{E}[|x_{t+1}-x^*|^2])\\
			&-\frac{\mu}{2}\mathbb{E}[|x_t-x^*|^2]+\frac{\alpha_t}{2}|\nabla f_t(x_{t})|^2.
		\end{split}
	\end{equation*}
	Since the gradient of $f_t$ is uniformly bounded by $G$, the following inequality holds
	\begin{equation}\label{scsr}
		\begin{split}
			R_T^s\leq& \sum_{t=1}^T\left[\frac{P}{2\alpha_t}(\mathbb{E}[|x_t-x^*|^2]-\mathbb{E}[|x_{t+1}-x^*|^2])\right]\\
			&-\sum_{t=1}^T\frac{\mu }{2}\mathbb{E}[|x_t-x^*|^2]+\sum_{t=1}^T\frac{\alpha_t}{2}|\nabla f_t(x_{t})|^2.
		\end{split}
	\end{equation}

	By choosing $\alpha_t=\frac{P}{\mu t}$, we have the following relationship
	\begin{equation}\label{scsr2}
		\begin{split}
			&R_T^s\leq0+\sum_{t=2}^T(\frac{P}{2\alpha_t}-\frac{P}{2\alpha_{t-1}}-\frac{\mu}{2})\mathbb{E}[|x_t-x^*|^2]\\
			&-\frac{\mu T}{2}\mathbb{E}[|x_{T+1}-x^*|^2]+\sum_{t=1}^T\frac{\alpha_tG^2}{2}\\
			&\leq 0+\sum_{t=1}^T\frac{\alpha_tG^2}{2}\leq \frac{PG^2}{2\mu}(1+\log T),
		\end{split}
	\end{equation}
	\color{black}where the second inequality follows by expanding the sum. Thus, the proof is complete. \hfill $\qed$
\end{pf}

By making the extra assumption of strong convexity, the sublinear static regret bound of order $O(\sqrt{T})$ established in Theorem \ref{T1-Ch4} is improved to be of order $O(\log T)$ which is consistent with the regret bound for online gradient descent algorithms under the same assumptions established in \cite{hazan2007logarithmic}.

\subsection{Dynamic regret for convex functions}
Now, we provide an upper bound of the dynamic regret of the online coordinate descent algorithm. First, we define the following measure of variations of the problem (\ref{MainProblem}) which is commonly used in the literature \cite{mokhtari2016online,shahrampour2018distributed}
\begin{equation}\label{measureCT}
	C_T:=\sum_{t=1}^T|x_t^*-x_{t-1}^*|,
\end{equation}
where $x_t^*$ is a solution to the optimization problem at time $t$ and $x_0^*$ can be arbitrary constant. The term $C_T$ captures the accumulated variations of optimal points at two consecutive time instances over time. If the cost functions are general convex functions and an upper bound of $C_T$ is known, the following results can be stated.
\begin{thm}\label{T3-Ch4}
	Suppose Assumption \ref{a1-Ch4} holds. Furthermore, if the stepsize is chosen as $\alpha_t=\sqrt{\frac{C_T}{T}}$, the dynamic regret (\ref{DRR}) achieved by Algorithm \ref{alg1-Ch4} satisfies
	\begin{equation*}
		R_T^d\leq (\frac{5R^2}{2\sqrt{C_T}}+RP\sqrt{C_T})\sqrt{T}+\frac{\sqrt{C_T}G^2}{2}\sqrt{T}.
	\end{equation*} 
\end{thm}

\begin{pf}
\color{black}Substituting $x_t^*$ in (\ref{coreinq}) for any $t\geq 2$\color{black}, yields
\begin{equation*}\label{DRcoreIP1}
	\begin{split}
		\mathbb{E}[f_t(x_t)-f_t(x_t^*)]\leq& \frac{P}{2\alpha_t}(\mathbb{E}[|x_t-x_t^*|^2]-\mathbb{E}[|x_{t+1}-x_t^*|^2])\\
		&+\frac{\alpha_t}{2}|\nabla f_t(x_{t})|^2.
	\end{split}
\end{equation*}
\color{black}Furthermore\color{black},
\begin{equation}\label{DRcoreIP2}
	\begin{split}
		&\sum_{t=1}^T(|x_t-x_t^*|^2-|x_{t+1}-x_t^*|^2)\leq\\
		&|x_1-x_1^*|^2+\sum_{t=2}^T(|x_t-x_t^*|^2-|x_{t}-x_{t-1}^*|^2)\\
		&=|x_1-x_1^*|^2+\sum_{t=2}^T2x_t^T(x_{t-1}^*-x_t^*)+\sum_{t=2}^T(|x_t^*|^2-|x_{t-1}^*|^2)\\
		&\leq |x_1-x_1^*|^2+|x_T^*|^2+2\sum_{t=2}^T2|x_t||x_{t-1}^*-x_t^*|\\
		&\leq 5R^2+2RC_T.\\
	\end{split}
\end{equation}
Then, 
\begin{equation*}\label{DRcoreIP3}
	R_T^d\leq \frac{P}{2\alpha_t}(5R^2+2RC_T)+\frac{\alpha_t}{2}G^2T.
\end{equation*}
\color{black}Choosing the stepsize to be $\sqrt{\frac{C_T}{T}}$ leads to\color{black}
\begin{equation*}\label{DRcoreIP4}
	R_T^d\leq (\frac{5R^2}{2\sqrt{C_T}}+RP\sqrt{C_T})\sqrt{T}+\frac{\sqrt{C_T}G^2}{2}\sqrt{T}.
\end{equation*} \hfill $\qed$
\end{pf}

\begin{remark}\label{DR-Conv}
	As discussed in Remark \ref{doubling}, the choice of stepsize $\sqrt{\frac{C_T}{T}}$ can be replaced by the doubling trick scheme (see Definition \ref{DTS}) to derive a regret bound of the same order. Moreover, the exact value of $C_T$ is also not necessary to derive the regret bound given in Theorem \ref{T3-Ch4}\color{black}. If $C_T$ is not available, a stepsize of the same order of $\sqrt{\frac{C_T}{T}}$ can be used with constant factor errors. We refer readers to \cite[Theorem 1]{cao2021online} for a more detailed discussion on how to implement stepsizes of the form $\sqrt{\frac{C_T}{T}}$ and obtain an upper bound of $C_T$ in practice using limited information. The inclusion of $C_T$ in the dynamic regret bound is common in existing literature \cite{cao2021online,shahrampour2018distributed,yi2023regret}. As argued in \cite{cao2021online}, if $C_T$ is sublinear in $T$, then the overall dynamic regret bound in Theorem \ref{T3-Ch4} will be sublinear in $T$, which is desired and consistent with the dynamic regret bounds established in \cite{cao2021online} for full-gradient based online algorithms. To see this, if the variation of minimizers decreases with the order ${1}/{t}$, then $C_T=O(\log T)$ and $R_T^d\leq O(\sqrt{T\log T})$. Similarly, if the variation of minimizers decreases with the order ${1}/{t^q}$ with $0<q<1$, then $C_T=O(T^{1-q})$ and $R_T^d\leq O(T^{1-\frac{q}{2}})$. In the worst case, by Assumption \ref{a1-Ch4}, we have $C_T\leq RT$ and $R_T^d\leq O(\sqrt{R}T)$. Thus, ${R_T^d}/{T}\leq O(\sqrt{R})$, which means the corresponding online algorithm incurs a steady-state tracking error. Moreover, the error decreases with the constant bound on the variation of minimizers.\color{black}
\end{remark}

\subsection{Dynamic regret for strongly convex functions}
Now, we consider the dynamic regret of the online random coordinate descent algorithm for strongly convex functions. As before, we consider $\mu$-strongly convex functions\color{black}. In addition, we will make the following assumption commonly seen in online optimization literature \cite{li2021online,mokhtari2016online}.
\begin{assumption}\label{LipGrad}
	For any $t\in\mathbb{Z}_{\geq 0}$, the gradient of $f_t$ is uniformly Lipschitz continuous, i.e., there exists $L>0$ such that the following inequality holds for any $t\in\mathbb{Z}_{\geq 0}$, $x,y\in\mathbb{R}^n$
	\begin{equation}\label{LipG}
		|\nabla f_t(x)-\nabla f_t(y)|\leq L|x-y|.
	\end{equation}	
\end{assumption}
It is shown in \cite[Proposition 6.1.2]{Bert1} that under the same conditions stated in Assumption \ref{LipGrad}, \eqref{LipG} is equivalent to 
\begin{equation}
	\label{LipG2}
	f_t(y)\leq f_t(x)+\nabla f_t(x)^T(y-x)+\frac{L}{2}|y-x|^2.
\end{equation}
\color{black}

The main result regarding the dynamic regret bound of Algorithm \ref{alg1-Ch4} for strongly convex functions with Lipschitz gradients is stated as follows.
\begin{thm}\label{T4-Ch4}
	Suppose $\nabla f_t(x_t^*)=0$ for any $t\in\mathbb{Z}_{\geq0}$, \color{black} Assumption \ref{a1-Ch4} (i) and Assumptions \ref{a2-Ch4}-\ref{LipGrad} hold\color{black}. If the stepsize is chosen as $\alpha_t=\alpha\leq\frac{2}{\mu+L}$, the dynamic regret (\ref{DRR}) achieved by Algorithm \ref{alg1-Ch4} satisfies
	\begin{equation*}
	R_T^d\leq \frac{G}{e}(C_T+C_1),
	\end{equation*} 
	where $e=1-\sqrt{1-\frac{2\alpha}{P}\frac{\mu L}{\mu+L}}$, $C_T=\sum_{t=1}^T|x_t^*-x_{t-1}^*|$, and $C_1=|x_1-x_1^*|-|x_1^*-x_0^*|$.
\end{thm}
\begin{pf}
	From (\ref{StrC}) and (\ref{LipG}), the following inequality holds \cite[Proposition 6.1.9]{Bert1} for any $t\in\mathbb{Z}_{\geq0}, x,y\in\mathbb{R}^n$
	\begin{equation}
		\label{p619}
	\begin{split}	
	&(\nabla f_t(x)-\nabla f_t(y))^T(x-y)\geq \\
	&\frac{\mu L}{\mu+L}|x-y|^2+\frac{1}{\mu+L}|\nabla f_t(x)-\nabla f_t(y)|^2.
    \end{split}
	\end{equation} 
    
	Then, using Lemma \ref{Proj_nonexp}, we have  $|x_{t+1}-x_t^*|^2\leq|x_{t}-\alpha U_t\nabla f_t(x_t)-x_t^*|^2$. Next, we take the conditional expectation given past history up to time $t$,
	\begin{equation}
		\label{EXP-priorexpan}
		\begin{split}
			&\mathbb{E}[|x_{t+1}-x_t^*|^2|\mathcal{F}_t]\leq-\frac{2\alpha}{P}(\nabla f_t(x_t)-\nabla f_t(x_t^*))^T(x_t-x_t^*)\\
			&+|x_{t}-x_t^*|^2+\alpha^2\mathbb{E}[|U_t\nabla f_t(x_t)|^2|\mathcal{F}_t].
		\end{split}
	\end{equation}
	The last term in \eqref{EXP-priorexpan} can be upper bounded using the following inequality 
	\begin{equation}
		\label{EXP-expan_last}
			\mathbb{E}[|U_t\nabla f_t(x_t)|^2|\mathcal{F}_t]{\leq}\frac{1}{P}|\nabla f_t(x_t)-\nabla f_t(x_t^*)|^2.
	\end{equation}
	The inequality \eqref{EXP-expan_last} comes from the fact that 
	\begin{equation*}
		\mathbb{E}[|U_t\nabla f_t(x_t)|^2|\mathcal{F}_t]=|\nabla f_t(x_t)|^2/P.
	\end{equation*}
	Combining \eqref{p619}, \eqref{EXP-priorexpan} and \eqref{EXP-expan_last}, we have for any $t\in\mathbb{Z}_{\geq0}$
	\begin{equation}\label{IneL10}
	\begin{split}
	&\mathbb{E}[|x_{t+1}-x_t^*|^2|\mathcal{F}_t]\leq(1-\frac{2\alpha}{P}\frac{\mu L}{\mu+L})|x_{t}-x_t^*|^2\\
	&+(\frac{\alpha^2}{P}-\frac{2\alpha}{P}\frac{1}{\mu+L})|\nabla f_t(x_t)-\nabla f_t(x_t^*)|^2\\
	&{\leq}(1-\frac{2\alpha}{P}\frac{\mu L}{\mu+L})|x_{t}-x_t^*|^2,
	\end{split}
	\end{equation}
    where the last inequality comes from the constant stepsize choice $\alpha\leq 2/(\mu+L)$. By Jensen's inequality, we have 
	\begin{equation}\label{compa}
	\mathbb{E}|x_{t+1}-x^*_t|\mathcal{F}_t|\leq \sqrt{1-\frac{2\alpha}{P}\frac{\mu L}{\mu+L}}|x_t-x_t^*|.
	\end{equation}
	\color{black}Note \color{black} that $\sqrt{1-\frac{2\alpha}{P}\frac{\mu L}{\mu+L}}<1$. Thus, there exists $0<e<1$ such that $\sqrt{1-\frac{2\alpha}{P}\frac{\mu L}{\mu+L}}=1-e$. 
	
	Assumption \ref{a1-Ch4} (i) implies $|f_t(x)-f_t(y)|\leq G|x-y|$ for any $t\in\mathbb{Z}_{\geq0}$ and any $x,y\in\Theta$. As a result, the stochastic dynamic regret is bounded as $R_T^d\leq G\mathbb{E}[\sum_{t=1}^{T}|x_t-x_t^*|]$. Moreover, we have
	\begin{equation}\label{compa2}
	|x_{t+1}-x_{t+1}^*|\leq|x_{t+1}-x_t^*|+|x_{t}^*-x_{t+1}^*|.
	\end{equation}
	Now taking total expectations, using (\ref{compa}) and summing up both sides of (\ref{compa2}) from $t=1$ to $t=T-1$ yields
	\begin{equation}\label{2ndtolast}
	R_T^d/G\leq |x_1-x_1^*|+(1-e)R_T^d/G+\sum_{t=1}^{T}|x_{t}^*-x_{t+1}^*|.
	\end{equation}
	Rearranging (\ref{2ndtolast}) leads to 
	\begin{equation*}
	R_T^d\leq \frac{G}{e}(C_T+C_1).
	\end{equation*}\hfill $\qed$

\end{pf}

\begin{remark}\label{R-T-4.4}
	By assuming uniform strong convexity and uniform Lipschitz continuity of gradients of $f_t$, Theorem \ref{T4-Ch4} improves the order of the theoretical dynamic regret bound from $O(\sqrt{C_TT})$ in Theorem \ref{T3-Ch4} to $O(C_T)$, which is consistent with results shown in \cite{mokhtari2016online} for full-gradient based online algorithms. This means if problem (\ref{MainProblem}) does not change over time i.e. $f_t=f$ for any $t$, $C_T=0$ and the regret grows at a rate $O(1)$ which is consistent with the convergence result of coordinate descent algorithms in the offline setting \cite{Wright}\color{black}. As item (ii) of Assumption \ref{a1-Ch4} is not required for proving Theorem \ref{T4-Ch4}, the regret bound is applicable to unconstrained problems with bounded gradients.
\end{remark}
\section{\color{black}Regret bounds for online coordinate descent algorithms with deterministic coordinate selection rules\color{black}}\label{RB-D}
We consider two deterministic online coordinate descent algorithms in this section and they are given in Algorithms \ref{alg2-Ch4} and \ref{alg3-Ch4}, respectively.
\begin{algorithm}
	\caption{Online cyclic coordinate descent algorithm}\label{alg2-Ch4}
	\begin{algorithmic}[1]
		\State \textbf{Initialization}:  $x_0, i_0$.
		\State \textbf{Select Coordinate}:  $i_t\leftarrow(i_{t-1}\mod P)+1$.
		\State \textbf{Update}:  For $i=i_t$:
		\begin{align*}
		x_{(i),t+1}\leftarrow x_{(i),t}-\alpha_t[\nabla_{(i)} f_{t}(x_{t})].
		\end{align*}
		For all $i\neq i_t$:
		\begin{equation*}
		x_{(i),t+1}\leftarrow x_{(i),t}.
		\end{equation*}
	\State \textbf{Projection}:  \color{black} $x_{t+1}\leftarrow\Pi_{\Theta}(x_{t+1})$.\color{black}	
		\State Set $t\leftarrow t+1$ and go to Step 2.
	\end{algorithmic}
\end{algorithm}

\begin{algorithm}
	\caption{Online coordinate descent algorithm with Gauss-Southwell Rule}\label{alg3-Ch4}
	\begin{algorithmic}[1]
		\State \textbf{Initialization}:  $x_0\in\Theta$.
		\State \textbf{Select Coordinate}: $i_t\leftarrow\arg \max_i|\nabla_{(i)} f_t(x_t)|$.
		\State \textbf{Update}:  For $i=i_t$:
		\begin{align*}
		x_{(i),t+1}\leftarrow x_{(i),t}-\alpha_t[\nabla_{(i)} f_{t}(x_{t})].
		\end{align*}
		For all $i\neq i_t$:
		\begin{equation*}
		x_{(i),t+1}\leftarrow x_{(i),t}.
		\end{equation*}
	\State \textbf{Projection}:  \color{black}$x_{t+1}\leftarrow\Pi_{\Theta}(x_{t+1})$.\color{black}	
		\State Set $t\leftarrow t+1$ and go to Step 2.
	\end{algorithmic}
\end{algorithm}

Note that both Algorithms \ref{alg2-Ch4} and \ref{alg3-Ch4} are deterministic and update a component of the decision variable. We can therefore relate the regrets achievable by Algorithms \ref{alg2-Ch4} and \ref{alg3-Ch4} to regret achievable by the online projected gradient descent algorithm which takes the form
\begin{equation}\label{OGD}
\color{black}x_{t+1}=\Pi_{\Theta}(x_{t}-\alpha_t\nabla f_t(x_{t})).\color{black}
\end{equation}
It can be easily seen that Algorithms \ref{alg2-Ch4} and \ref{alg3-Ch4} take the following form\color{black} 
\begin{equation}\label{OGD-e}
x_{t+1}=\Pi_{\Theta}(x_{t}-\alpha_t(\nabla f_t(x_{t})+\nu_t)),
\end{equation}
\color{black}where the vector $\nu_t:=\overline{\nabla}_{(i_t)} f_t(x_{t})-\nabla f_t(x_{t})$ with $\overline{\nabla}_{(i_t)} f_t(x_{t})\in\mathbb{R}^n$ denotes the vector such that its $i_t$-th component is ${\nabla}_{(i_t)} f_t(x_{t})$ while all other entries are $0$\color{black}. It captures the effect of the components of the gradient that are not updating. We use $\bar{R}_T^s$ and $\bar{R}_T^d$ to denote the static and dynamic regrets of the online projected gradient descent algorithm respectively. Then, we can have the following result.
\begin{proposition}\label{P1-Ch4}
	Suppose Assumption \ref{a1-Ch4} holds, then the static regret $R_T^s$ and dynamic regret $R_T^d$ of \color{black} iterations \color{black} (\ref{OGD-e}) satisfy the following two relationships.
	\begin{enumerate}
		\item $R_T^s\leq \bar{R}_T^s+G^2 \sum_{t=1}^{T}\alpha_t$.
		\item $R_T^d\leq \bar{R}_T^d+G^2 \sum_{t=1}^{T}\alpha_t$.
	\end{enumerate}
\end{proposition}
\begin{pf}
	By the definitions of regrets in (\ref{S-static-R}) and (\ref{S-dynamic-R}), it is obvious that $R_T^s-\bar{R}_T^s=R_T^d-\bar{R}_T^d$. Thus, if one of the two items is proved, so is the other item. Since Assumption \ref{a1-Ch4} holds\color{black}, from (\ref{S-static-R}) and Lemma \ref{Proj_nonexp}\color{black}, we know that $R_T^s-\bar{R}_T^s\leq \sum_{t=1}^{T} G|\alpha_t \nu|\leq G^2\sum_{t=1}^{T} |\alpha_t|=G^2 \sum_{t=1}^{T}\alpha_t$. Thus the proof is complete. \hfill $\qed$
\end{pf}
By using Proposition \ref{P1-Ch4}, we can establish regret bounds for Algorithms \ref{alg2-Ch4} and \ref{alg3-Ch4} using known regret bounds for online gradient descent algorithms. Moreover, if the regret bounds for online gradient descent are sublinear in $T$ and $\sum_{t=1}^{T}\alpha_t$ is also sublinear in $T$, then the established regret bounds for Algorithms \ref{alg2-Ch4} and \ref{alg3-Ch4} will be sublinear.
\subsection{Static regret for convex functions}
The following static regret bound of online projected gradient descent algorithm is proved in \cite[Theorem 1]{zinkevich2003online}.
\begin{lemma}\label{L1-CH4}
	Suppose Assumption \ref{a1-Ch4} holds. Furthermore, if the stepsize is chosen as $\alpha_t=\sqrt{\frac{1}{t}}$, then the static regret of online projected gradient descent \color{black} iterations (\ref{OGD}) satisfy\color{black}
	\begin{equation*}
	\bar{R}_T^s\leq \frac{R^2\sqrt{T}}{2}+(\sqrt{T}-\frac{1}{2})G^2.
	\end{equation*}
\end{lemma}
The following result on static regret bounds of Algorithms \ref{alg2-Ch4} and \ref{alg3-Ch4} is a direct corollary of Proposition \ref{P1-Ch4} and Lemma \ref{L1-CH4}.
\begin{corollary}\label{C1-Ch4}
	Suppose Assumption \ref{a1-Ch4} holds. Furthermore, if the stepsize is chosen as $\alpha_t=\sqrt{\frac{1}{t}}$, then the static regrets of Algorithms \ref{alg2-Ch4} and \ref{alg3-Ch4} satisfy
	\begin{equation*}
	R_T^s\leq \frac{R^2\sqrt{T}}{2}+(\sqrt{T}-\frac{1}{2})G^2+2G^2\sqrt{T}.
	\end{equation*}
\end{corollary}
\begin{pf}
	Since $\alpha_t=\sqrt{\frac{1}{t}}$, we have
	\begin{equation*}
	\begin{split}
	\sum_{t=1}^{T}\alpha_t&=1+\sum_{t=2}^{T}\alpha_t\\
	&\leq 1+\int_{s=1}^{T}\frac{1}{\sqrt{s}}ds\\
	&=2\sqrt{T}.
	\end{split}
	\end{equation*}
	The conclusion follows as a result of Proposition \ref{P1-Ch4} and Lemma \ref{L1-CH4}. \hfill $\qed$
\end{pf}

\subsection{Static regret for strongly convex functions}
The static regret bound in Lemma \ref{L1-CH4} is improved in \cite[Theorem 1]{hazan2007logarithmic} under the assumption of strong convexity.
\begin{lemma}\label{L2-CH4}
	Suppose Assumptions \ref{a1-Ch4} and \ref{a2-Ch4} hold. Furthermore, if the stepsize is chosen as $\alpha_t=\frac{1}{\mu t}$, then the static regret of online projected gradient descent algorithm (\ref{OGD}) satisfies
	\begin{equation*}
	\bar{R}_T^s\leq \frac{G^2}{2\mu}(1+\log T).
	\end{equation*}
\end{lemma}
The following result on static regret bounds of Algorithms \ref{alg2-Ch4} and \ref{alg3-Ch4} for strongly convex functions is a direct corollary of Proposition \ref{P1-Ch4} and Lemma \ref{L2-CH4}.
\begin{corollary}\label{C2-Ch4}
	Suppose Assumptions \ref{a1-Ch4} and \ref{a2-Ch4} hold. Furthermore, if the stepsize is chosen as $\alpha_t=\frac{1}{\mu t}$, then the static regrets of Algorithms \ref{alg2-Ch4} and \ref{alg3-Ch4} satisfy
	\begin{equation*}
	R_T^s\leq \frac{3G^2}{2\mu}(1+\log T).
	\end{equation*}
\end{corollary}
\begin{pf}
	Since $\alpha_t=\frac{1}{\mu t}$, we have
	\begin{equation*}
	\begin{split}
	\sum_{t=1}^{T}\alpha_t&=1+\sum_{t=2}^{T}\alpha_t\\
	&\leq 1+\int_{s=1}^{T}\frac{1}{\mu s}ds\\
	&=1+\log T.
	\end{split}
	\end{equation*}
	The conclusion follows as a result of Proposition \ref{P1-Ch4} and Lemma \ref{L2-CH4}. \hfill $\qed$
\end{pf}
\subsection{Dynamic regret for convex functions}
We can adapt the proof of Theorem \ref{T3-Ch4} to the deterministic case to derive a dynamic regret bound for convex functions which is given in the following result.
\begin{proposition}\label{P-T3-Ch4}
	Suppose Assumption \ref{a1-Ch4} holds. Furthermore, if the stepsize is chosen as $\alpha_t=\sqrt{\frac{C_T}{T}}$, the dynamic regret achieved by the online gradient descent algorithm (\ref{OGD}) satisfies
	\begin{equation*}
	\bar{R}_T^d\leq (\frac{5R^2}{2\sqrt{C_T}}+R\sqrt{C_T})\sqrt{T}+\frac{\sqrt{C_T}G^2}{2}\sqrt{T}.
	\end{equation*} 
\end{proposition}

\begin{pf}
	\color{black}From (\ref{OGD}) and Lemma \ref{Proj_nonexp}, we have
	\begin{equation}\label{coreinq-OGD}
	\begin{split}
	&|x_{t+1}-x|^2\leq|x_{t}-\alpha_t\nabla f_t(x_{t})-x|^2\\
	&=|x_t-x|^2-2\alpha_t\nabla f_t(x_{t})^T(x_t-x)+\alpha_t^2|\nabla f_t(x_{t})|^2
	\end{split}
	\end{equation}
	 for any $x\in\Theta$\color{black}. Substituting $x_t^*$ to (\ref{coreinq-OGD}) results in the following inequality for any $t\geq 2$
	\begin{equation*}
	\begin{split}
	&f_t(x_t)-f_t(x_t^*)\leq\\ &\frac{1}{2\alpha_t}(|x_t-x_t^*|^2-|x_{t+1}-x_t^*|^2)+\frac{\alpha_t}{2}|\nabla f_t(x_{t})|^2.
	\end{split}
	\end{equation*}
	Note that, the inequality (\ref{DRcoreIP2}) still holds in this case
	Thus, we have 
	\begin{equation*}
	\bar{R}_T^d\leq \frac{1}{2\alpha_t}(5R^2+2RC_T)+\frac{\alpha_t}{2}G^2T.
	\end{equation*}
	If we set the stepsize to be $\sqrt{\frac{C_T}{T}}$, we have
	\begin{equation*}
	\bar{R}_T^d\leq (\frac{5R^2}{2\sqrt{C_T}}+R\sqrt{C_T})\sqrt{T}+\frac{\sqrt{C_T}G^2}{2}\sqrt{T}.
	\end{equation*} \hfill $\qed$
\end{pf}

Using Propositions \ref{P1-Ch4} and \ref{P-T3-Ch4}, we can state the following result on dynamic regret bounds of Algorithms \ref{alg2-Ch4} and \ref{alg3-Ch4}.
\begin{corollary}\label{C3-Ch4}
	Suppose Assumption \ref{a1-Ch4} holds. Furthermore, if the stepsize is chosen as $\alpha_t=\sqrt{\frac{C_T}{T}}$, then the dynamic regrets of Algorithms \ref{alg2-Ch4} and \ref{alg3-Ch4} satisfy
	\begin{equation*}
	R_T^d\leq \frac{R^2\sqrt{T}}{2}+(\sqrt{T}-\frac{1}{2})G^2+2G^2\sqrt{C_TT}.
	\end{equation*}
\end{corollary}
\begin{pf}
	Since $\alpha_t=\sqrt{\frac{C_T}{T}}$, we have $\sum_{t=1}^{T}\alpha_t=\sqrt{C_TT}$. The conclusion follows as a result of Propositions \ref{P1-Ch4} and \ref{P-T3-Ch4}. \hfill $\qed$
\end{pf}
As discussed in Remark \ref{DR-Conv}, the stepsize choice $\alpha_t=\sqrt{\frac{C_T}{T}}$ can be made independent of $T$ by using the doubling trick scheme.
\subsection{Dynamic regret for strongly convex functions}
\color{black}Note that when the cost functions are strongly convex and constant stepsizes are used, the result in Proposition \ref{P1-Ch4} only gives a dynamic regret bound that is linear in $T$. Therefore, we aim to establish better dynamic regret bounds for Algorithms \ref{alg2-Ch4} and \ref{alg3-Ch4}\color{black}. In this subsection, we will make the following assumption on the problem (\ref{MainProblem}).
\begin{assumption}\label{a3-CH4}
	The gradient of $f_t$ is block-wise Lipschitz continuous uniformly in $t$, i.e., for any $1\leq i\leq P$, there exists $L_i>0$ such that the following inequality holds for any $t\in\mathbb{Z}_{\geq 0}$, $x\in\mathbb{R}^n$, and $u\in\mathbb{R}^{n_i}$
	\begin{equation}\label{LipG-block}
	|\nabla_{(i)} f_t(x+H_{i}u)-\nabla_{(i)} f_t(x)|\leq L_i|u|,
	\end{equation}
	where $H_{i}\in\mathbb{R}^{n\times n_i}$ is such that $[H_1\ H_2\cdots H_P]=I_{n}$.	
\end{assumption}

We denote the largest $L_i$ for $1\leq i\leq P$ by $L_{\mathrm{max}}:={\max}\{L_1,\cdots,L_P\}$\color{black}. By further assuming $\nabla f_t(x_t^*)=0$ for any $t\in\mathbb{Z}_{\geq0}$, the dynamic regret bounds can be derived without Assumption \ref{a1-Ch4}. Therefore, they are applicable to unconstrained problems with unbounded gradients. Before analyzing the regrets, we define the following measure of variations of the problem (\ref{MainProblem}) which is first introduced in \cite{zhang2017improved},
\begin{equation}\label{measureCT2}
	C_{T,2}:=\sum_{t=1}^T|x_t^*-x_{t-1}^*|^2.
\end{equation}
\begin{remark}
	\label{remark-C_T_2}
	\color{black}Compared to $C_T$, $C_{T,2}$ can be significantly smaller if the variation of minimizers is small. Following the discussion in Remark \ref{DR-Conv}, if the variation of minimizers decreases with the order ${1}/{t^q}$ with $1/2<q\leq1$, then $C_{T,2}$ is finite while $C_T$ grows to $\infty$ with a rate sublinear in $T$. Similarly, if $0<q<1/2$, then $C_{T,2}=O(T^{1-2q}), C_T=O(T^{1-q})$.\color{black} 
\end{remark}

\begin{thm}\label{T6-Ch4}
	Suppose $\nabla f_t(x_t^*)=0$ for any $t\in\mathbb{Z}_{\geq0}$ and Assumptions \ref{a2-Ch4}-\ref{a3-CH4} hold. Furthermore, we assume the number of blocks $P$ satisfies $P<\frac{1}{B_3L_{\mathrm{max}}}$ with $B_3:=\frac{1}{\mu}-\frac{1}{2L}>0$. If the stepsize is chosen such that $\alpha_t=\alpha\in(\frac{1-\sqrt{1-B_3PL_{\mathrm{max}}}}{L_{\mathrm{max}}},\frac{1+\sqrt{1-B_3PL_{\mathrm{max}}}}{L_{\mathrm{max}}})$. Then, the dynamic regret (\ref{DRR}) achieved by Algorithm \ref{alg3-Ch4} satisfies
	\begin{equation*}
		R_T^d\leq \frac{L}{B_4}(C_{T,2}+C_2),
	\end{equation*} 
	where $B_4=\frac{1}{2}-\sqrt{L(\frac{1}{\mu}-\frac{2\alpha-\alpha^2 L_{\mathrm{max}}}{2P})}>0$ and $C_2=|x_1-x_1^*|^2-2|x_1^*-x_0^*|^2$.
\end{thm}
\begin{pf}
	\color{black}Let $\{y_t\}$ denote the sequence of real vectors such that $y_{t+1}=x_{t}-\alpha_tU_t[\nabla f_t(x_{t})]$, where the value of $U_t$ follows the coordinate selection rule in Algorithm \ref{alg3-Ch4}. The variable $y_t\in\mathbb{R}^n$ stores the value of the decision variable $x_t$ before projection, i.e., $x_t=\Pi_{\Theta}(y_t)$. By the block descent lemma \cite[Lemma 3.2]{beck2013convergence} and the fact that $L_{\mathrm{max}}\geq L_i$ for all $i$, we have $f_t(y_{t+1})\leq f_t(x_{t})+\frac{\alpha^2 L_{\mathrm{max}}}{2}|\nabla_{(i)}f_t(x_{t})|^2-\alpha|\nabla_{(i)} f_t(x_{t})|^2$. That is 
	\begin{equation}\label{blockdescent}
		f_t(x_{t})-f_t(y_{t+1})\geq (\alpha-\frac{\alpha^2 L_{\mathrm{max}}}{2})|\nabla_{(i)} f_t(x_{t})|^2,
	\end{equation}
	From (\ref{blockdescent}) and $i_t\in\arg \max_i|\nabla_{(i)} f_t(x_t)|$: 
	\begin{equation}\label{blockdescent-alg3}
		\begin{split}
		f_t(x_{t})-f_t(y_{t+1})&\geq (\alpha-\frac{\alpha^2 L_{\mathrm{max}}}{2})|\nabla_{(i)} f_t(x_{t})|^2\\
		&\geq\frac{1}{P}(\alpha-\frac{\alpha^2 L_{\mathrm{max}}}{2})|\nabla f_t(x_{t})|^2,
	\end{split}
	\end{equation}
	Since $\nabla f_t(x_t^*)=0$ for any $t\in\mathbb{Z}_{\geq0}$, minimizing both sides of \eqref{StrC} with respect to $y$, we have $f_t(x_t)-f_t(x_t^*)\leq \frac{1}{2\mu}|\nabla f_t(x_t)|^2$ for any $x_t\in\Theta$ (known as the Polyak-\L ojasiewicz condition). Then, by (\ref{blockdescent-alg3}), we have
	\begin{equation*}
		\begin{split}
		f_t(x_{t})-f_t(x_t^*)-&(f_t(y_{t+1})-f_t(x_t^*))\geq \bar{A}|\nabla f_t(x_{t})|^2\\
		&\geq 2\mu \bar{A} (f_t(x_{t})-f_t(x_t^*)),
		\end{split}
	\end{equation*}
	where $\bar{A}=\frac{1}{P}(\alpha-\frac{\alpha^2 L_{\mathrm{max}}}{2})$. Consequently,
	\begin{equation}\label{Convbound-alg3}
		f_t(y_{t+1})-f_t(x_t^*)\leq (1-2\mu \bar{A})(f_t(x_{t})-f_t(x_t^*)).
	\end{equation}
	By the fact that $\nabla f_t(x_t^*)=0$ and Assumption \ref{LipGrad}, we have $f_t(x_t)-f_t(x_t^*)\leq \frac{L}{2}|x_t-x_t^*|^2$ and hence
	\begin{equation}\label{Lybounds-alg2}
		\frac{\mu}{2}|x_t-x_t^*|^2\leq f_t(x_t)-f_t(x_t^*)\leq \frac{L}{2}|x_t-x_t^*|^2.
	\end{equation}
	By (\ref{Convbound-alg3}) and (\ref{Lybounds-alg2}), the following inequality holds as a result of Lemma \ref{Proj_nonexp}
	\begin{equation}\label{statebounds-alg3}
	|x_{t+1}-x_t^*|^2\leq |y_{t+1}-x_t^*|^2\leq \frac{L}{\mu}(1-2\mu \bar{A})|x_t-x_t^*|^2.
	\end{equation}
	
	Next we show that $\frac{L}{\mu}(1-2\mu \bar{A})<\frac{1}{2}$ under the stated assumptions. Since $\bar{A}=\frac{1}{P}(\alpha-\frac{\alpha^2 L_{\mathrm{max}}}{2})$, it can be shown that $\frac{L}{\mu}(1-2\mu \bar{A})<\frac{1}{2}$ is equivalent to $L_{\mathrm{max}}\alpha^2-2\alpha+B_3P<0$ after some algebraic manipulations. When $P<\frac{1}{B_3L_{\mathrm{max}}}$ we have $4-4B_3PL_{\mathrm{max}}>0$ and the set of solutions to the inequality $L_{\mathrm{max}}\alpha^2-2\alpha+B_3P<0$ with respect to $\alpha$ is non-empty. Moreover the solution is given by \color{black} $\frac{1-\sqrt{1-B_3PL_{\mathrm{max}}}}{L_{\mathrm{max}}}<\alpha<\frac{1+\sqrt{1-B_3PL_{\mathrm{max}}}}{L_{\mathrm{max}}}$\color{black}. Therefore, the listed conditions in the statement of the theorem ensures that $\frac{L}{\mu}(1-2\mu \bar{A})<\frac{1}{2}$. Thus, \eqref{statebounds-alg3} implies
	\begin{equation}\label{statebounds-alg3-B4}
		|x_{t+1}-x_t^*|^2\leq \frac{1}{2}|x_t-x_t^*|^2-B_4|x_t-x_t^*|^2,
	\end{equation}
	with $B_4=\frac{1}{2}-\frac{L}{\mu}(1-2\mu \bar{A})>0$.
	
	Note that
	    \begin{equation}\label{compa2_2_alg3}
		    	|x_{t+1}-x_{t+1}^*|^2\leq2|x_{t+1}-x_t^*|^2+2|x_{t}^*-x_{t+1}^*|^2.
		    \end{equation}
	The deterministic dynamic regret \eqref{S-dynamic-R} can be upper bounded as follows
	    \begin{equation}
		    	\label{reg-strong-c}
		    	\begin{split}
			    		R_T^d&=\sum_{t=1}^{T}f_t(x_t)-\sum_{t=1}^{T}f_t(x_t^*)\\
			    		&\overset{\eqref{LipG2}}{\leq} \sum_{t=1}^{T}\nabla f_t(x_t^*)^T(x_t-x_t^*)+\frac{L}{2}|x_t-x_t^*|^2\\
			    		&=\frac{L}{2}\sum_{t=1}^{T}|x_t-x_t^*|^2.
			    	\end{split}
		    \end{equation}
	From (\ref{statebounds-alg3-B4}), we have the following relationship by summing up both sides of (\ref{compa2_2_alg3}) from $t=1$ to $t=T-1$,
	\begin{equation}
		\label{reg-rearrange-T5}
		\begin{split}
			&\sum_{t=1}^{T}|x_t-x_t^*|^2-|x_1-x_1^*|^2\\
			&\leq(1-2B_4) \sum_{t=1}^{T}|x_t-x_t^*|^2+2C_{T,2}-2|x_1^*-x_0^*|^2.
		\end{split}
	\end{equation}
	Since $B_4>0$, \eqref{reg-rearrange-T5} implies leads to $R_T^d\leq \frac{L}{2}\sum_{t=1}^{T}|x_t-x_t^*|^2\leq\frac{L}{B_4}(C_{T,2}+C_2)$. \hfill $\qed$
\end{pf}

\begin{remark}\label{remark-T5-Ch4}
	Note that in Theorem \ref{T6-Ch4}, we introduce an upper bound that depends \color{black} on the number of components $P$ which increases to $\infty$ as $B_3$ decreases to $0$. This is mainly a result of conservatism introduced in the derivation of \eqref{statebounds-alg3}. However, we manage to improve the regret bound in Theorem \ref{T4-Ch4} from $O(C_T)$ to $O(C_{T,2})$ \color{black} following the discussion in Remark \ref{remark-C_T_2}\color{black}. The bound is also valid for unconstrained OCO problems that are not covered by Theorem \ref{T1-Ch4}-\ref{T4-Ch4}. 
\end{remark}

\color{black}The following two theorems give dynamic regret bounds without assuming an upper bound on the number of components $P$. However, at each time $t$, potentially multiple offline steps are needed to guarantee desirable regret bounds.

The modified algorithms are in Algorithm \ref{alg4-Ch4}.
\begin{algorithm}
	\caption{Online cyclic coordinate descent algorithm}\label{alg4-Ch4}
	\begin{algorithmic}[1]
		\State \textbf{Initialization}\color{black}:  $x_0\in\Theta, i_0,\kappa$ and some integer $k\geq1$\color{black}.
		\State \textbf{Update}: At time $t$, $\kappa\leftarrow 1, i_{\kappa-1}\leftarrow i_{t-1}$, $\hat{x}_{\kappa-1}\leftarrow x_{t}$.
		\State \textbf{Select Coordinate}: $i_{\kappa}\leftarrow(i_{\kappa-1}\mod P)+1$.
		\State \textbf{Update}: For $i=i_{\kappa}$, such that $\kappa\leq k$:
		\begin{align*}
		\hat{x}_{(i),\kappa}\leftarrow\hat{x}_{(i),\kappa-1}-\alpha_t[\nabla_{(i)} f_{t}(\hat{x}_{\kappa-1})].
		\end{align*}
		\quad \quad \quad \quad \ For all $i\neq i_{\kappa}$, such that $\kappa\leq k$:
		\begin{equation*}
		\hat{x}_{(i),\kappa}\leftarrow\hat{x}_{(i),\kappa-1}.
		\end{equation*}
		\State For $\kappa< k$, set $\kappa\leftarrow \kappa+1$ and go to Step 3.
		\State For $\kappa\geq k$, set $t\leftarrow t+1, i_t=i_k\color{black}, x_{t+1}\leftarrow\Pi_{\Theta}(\hat{x}_{k})$ and go to Step 2\color{black}.
	\end{algorithmic}
\end{algorithm}

\begin{algorithm}
	\caption{Online coordinate descent algorithm with Gauss-Southwell Rule}\label{alg5-Ch4}
	\begin{algorithmic}[1]
		\State \textbf{Initialization}\color{black}:  $x_0\in\Theta, \kappa$ and some integer $k\geq1$\color{black}.
		\State \textbf{Update}: At time $t$, $\kappa\leftarrow 1$, $\hat{x}_{\kappa-1}\leftarrow x_{t}$.
		\State \textbf{Select Coordinate}: $i_{\kappa}\leftarrow\arg \max_i|\nabla_{(i)} f_t(x_{\kappa-1})|$.
		\State \textbf{Update}: For $i=i_{\kappa}$, such that $\kappa\leq k$:
		\begin{align*}
			\hat{x}_{(i),\kappa}\leftarrow\hat{x}_{(i),\kappa-1}-\alpha_t[\nabla_{(i)} f_{t}(\hat{x}_{\kappa-1})].
		\end{align*}
		\quad \quad \quad \quad \ For all $i\neq i_{\kappa}$, such that $\kappa\leq k$:
		\begin{equation*}
			\hat{x}_{(i),\kappa}\leftarrow\hat{x}_{(i),\kappa-1}.
		\end{equation*}
		\State For $\kappa<k$, set $\kappa\leftarrow \kappa+1$ and go to Step 3.
		\State For $\kappa\geq k$, set $t\leftarrow t+1\color{black}, x_{t+1}\leftarrow\Pi_{\Theta}(\hat{x}_{k})$ and go to Step 2\color{black}.
	\end{algorithmic}
\end{algorithm}

It can be seen that in Algorithms \ref{alg4-Ch4} and \ref{alg5-Ch4}, at each time $t$, $k$ updates are performed where $k$ is an integer to be chosen. In Algorithms \ref{alg2-Ch4} and \ref{alg3-Ch4}, however, only one step is performed at each time $t$.
 

\begin{thm}\label{T7-Ch4}
	\color{black}Suppose $\nabla f_t(x_t^*)=0$ for any $t\in\mathbb{Z}_{\geq0}$ and Assumptions \ref{a2-Ch4}-\ref{a3-CH4} hold and the stepsize is chosen such that $\alpha_t=\alpha<\frac{2}{L_{\mathrm{max}}}$. Let $k$ be an integer such that $B_5:=\frac{L}{\mu}(1-2\mu A)^{\frac{k+1-P}{P}}(1+\alpha L_{\mathrm{max}})^{2P-2}<1/2$ holds, then the dynamic regret (\ref{DRR}) achieved by Algorithm \ref{alg4-Ch4} satisfies
	\begin{equation*}
	R_T^d\leq \frac{L}{2-4B_5}(2C_{T,2}+C_2),
	\end{equation*} 
	where $A=(\alpha-\frac{\alpha^2 L_{\mathrm{max}}}{2})/[2(1+\alpha^2 L^2P)]$ and $C_2=|x_1-x_1^*|^2-2|x_1^*-x_0^*|^2$.\color{black}
\end{thm}
\begin{pf}
	By the block descent lemma \cite[Lemma 3.2]{beck2013convergence} and the fact that $L_{\mathrm{max}}\geq L_i$ for all $i$, we have $f_t(\hat{x}_{\kappa+1})\leq f_t(\hat{x}_{\kappa})+\frac{\alpha^2 L_{\mathrm{max}}}{2}|\nabla_{(i)}f_t(\hat{x}_{\kappa})|^2-\alpha|\nabla_{(i)} f_t(\hat{x}_{\kappa})|^2$ at any $t$, for any $0\leq \kappa\leq k-1$. That is 
	\begin{equation*}
		f_t(\hat{x}_{\kappa})-f_t(\hat{x}_{\kappa+1})\geq (\alpha-\frac{\alpha^2 L_{\mathrm{max}}}{2})|\nabla_{(i)} f_t(\hat{x}_{\kappa})|^2,
	\end{equation*}
	Summing over all $P$ blocks in a full round where all components have been updated exactly once leads to
	\begin{equation*}
		f_t(\hat{x}_{\kappa})-f_t(\hat{x}_{\kappa+P})\geq (\alpha-\frac{\alpha^2 L_{\mathrm{max}}}{2})\sum_{i=1}^P|\nabla_{(i)} f_t(\hat{x}_{\kappa+i})|^2.
	\end{equation*}
	By the update equation and Lipschitz continuity of the gradient, we have $|\nabla f_t(\hat{x}_{\kappa})-\nabla f_t(\hat{x}_{\kappa+i})|^2\leq \alpha^2 L^2\sum_{j=1}^i|\nabla_{(j)} f_t(\hat{x}_{\kappa+j-1})|^2$. Therefore
	\begin{equation*}
		\begin{split}
			&|\nabla_{(i)} f_t(\hat{x}_{\kappa})|^2\\
			&\leq (|\nabla_{(i)} f_t(\hat{x}_{\kappa})-\nabla_{(i)} f_t(\hat{x}_{\kappa+i})|+|\nabla_{(i)} f_t(\hat{x}_{\kappa+i})|)^2\\
			&\leq 2|\nabla_{(i)} f_t(\hat{x}_{\kappa+i})|^2+2\alpha^2 L^2\sum_{j=1}^i|\nabla_{(j)} f_t(\hat{x}_{\kappa+j-1})|^2.
		\end{split}
	\end{equation*}
	Summing over $P$ blocks leads to
	\begin{equation*}
		\begin{split}
			\sum_{i=1}^P|\nabla_{(i)} f_t(\hat{x}_{\kappa})|^2 &\leq 2\sum_{i=1}^P(1+(P-i)\alpha^2 L^2)|\nabla_{(i)} f_t(\hat{x}_{\kappa+i})|^2\\
			&\leq 2(1+\alpha^2 P L^2)\sum_{i=1}^P|\nabla_{(i)} f_t(\hat{x}_{\kappa+i})|^2.\\
		\end{split}
	\end{equation*}
	Therefore, 
	\begin{equation}\label{Adefine}
		\begin{split}
			&f_t(\hat{x}_{\kappa})-f_t(\hat{x}_{\kappa+P})\\
			&\geq (\alpha-\frac{\alpha^2 L_{\mathrm{max}}}{2})/[2(1+\alpha^2 L^2P)]|\nabla f_t(\hat{x}_{\kappa})|^2\\
			&:= A|\nabla f_t(\hat{x}_{\kappa})|^2.\\
		\end{split}
	\end{equation}
	By Assumption \ref{a2-Ch4}, minimizing both sides of \eqref{StrC} with respect to $y$, we have $f_t(\hat{x}_{\kappa})-f_t(x_t^*)\leq \frac{1}{2\mu}|\nabla f_t(\hat{x}_{\kappa})|^2$. Then, by (\ref{Adefine}), we have
	\begin{equation*}
		\begin{split}
		&f_t(\hat{x}_{\kappa})-f_t(x_t^*)-(f_t(\hat{x}_{\kappa+P})-f_t(x_t^*))\geq A|\nabla f_t(\hat{x}_{\kappa})|^2\\
		&\geq 2\mu A (f_t(\hat{x}_{\kappa})-f_t(x_t^*)),
	\end{split}
	\end{equation*}
	which further implies, 
	\begin{equation}\label{Convbound-alg2}
		f_t(\hat{x}_{\kappa+P})-f_t(x_t^*)\leq (1-2\mu A)(f_t(\hat{x}_{\kappa})-f_t(x_t^*)).
	\end{equation}
	By Assumption \ref{LipGrad}, we have $|\hat{x}_{\kappa+1}-x_{t}^*|\leq (1+\alpha L_{\mathrm{max}})|\hat{x}_{\kappa}-x_{t}^*|$. For any $k\in\mathbb{Z}_{\geq0}$, there exist $k_1\in\mathbb{Z}_{\geq0}$ and $k_2\in\{1,\ldots,P-1\}$ such that $k=k_1P+k_2$. Then, by Assumptions \ref{a2-Ch4} and \ref{LipGrad}, we have
	\begin{equation}\label{Lybounds-alg2-v2}
	\frac{\mu}{2}|x-x_t^*|^2\leq f_t(x)-f_t(x_t^*)\leq \frac{L}{2}|x-x_t^*|^2,
	\end{equation}
	for any $x\in\mathbb{R}^n$. \color{black}
	Consequently, (\ref{Convbound-alg2}) and (\ref{Lybounds-alg2-v2}) yield
	\begin{equation}\label{Comparisonbound-alg2}
	\begin{split}
	&|\hat{x}_{k}-x_t^*|^2= |\hat{x}_{k_1P+k_2}-x_{t}^*|^2\\
	&\leq \frac{2}{\mu}(1-2\mu A)^{k_1}\frac{L}{2}|\hat{x}_{k_2}-x_t^*|^2\\
	&={\frac{L}{\mu}}(1-2\mu A)^{k_1}|\hat{x}_{k_2}-x_t^*|^2\\
	&\leq {\frac{L}{\mu}}(1-2\mu A)^{\frac{k-k_2}{P}}(1+\alpha L_{\mathrm{max}})^{2P-2}|x_{t}-x_t^*|^2\\
	&\leq {\frac{L}{\mu}}(1-2\mu A)^{\frac{k+1-P}{P}}(1+\alpha L_{\mathrm{max}})^{2P-2}|x_{t}-x_t^*|^2,
	\end{split}
	\end{equation}
	at any $t\in\mathbb{Z}_{\geq 0}$. Since $\alpha<{2}/{L_{\mathrm{max}}}$, we have $A>0$ and $0<1-2\mu A<1$. Hence there exists $k$ such that $B_5=\frac{L}{\mu}(1-2\mu A)^{\frac{k+1-P}{P}}(1+\alpha L_{\mathrm{max}})^{2P-2}<1/2$. Hence, by Lemma \ref{Proj_nonexp}
	\begin{equation}\label{statebounds-alg2-v2}
	|x_{t+1}-x_t^*|^2\leq |\hat{x}_{k}-x_t^*|^2\leq B_5|x_t-x_t^*|^2.
	\end{equation}
	
	By \eqref{reg-strong-c}, we have $R_T^d\leq \frac{L}{2}\sum_{t=1}^{T}|x_t-x_t^*|^2$. Note that 
	\begin{equation}\label{compa2_1-v2}
	|x_{t+1}-x_{t+1}^*|^2\leq2|x_{t+1}-x_t^*|^2+2|x_{t}^*-x_{t+1}^*|^2.
	\end{equation}
	Summing up both sides of (\ref{compa2_1-v2}) from $t=1$ to $t=T-1$ using (\ref{statebounds-alg2-v2}), we have
	\begin{equation*}
	\begin{split}
	(1-2B_5)2R_T^d/L&\leq |x_1-x_1^*|^2+2\sum_{t=1}^{T}|x_{t}^*-x_{t+1}^*|^2\\
	&=2C_{T,2}+C_2.
	\end{split}
	\end{equation*}
	Since $2B_5<1$, $R_T^d\leq \frac{L}{2-4B_5}(2C_{T,2}+C_2)$ follows. \hfill $\qed$
\end{pf}

\begin{thm}\label{T8-Ch4}
	\color{black}Suppose $\nabla f_t(x_t^*)=0$ for any $t\in\mathbb{Z}_{\geq0}$ and Assumptions \ref{a2-Ch4}-\ref{a3-CH4} \color{black} hold and the stepsize is chosen such that $\alpha_t=\alpha<\frac{2}{L_{\mathrm{max}}}$. Let $k$ be an integer such that $B_6:={\frac{L}{\mu}}(1-2\mu \bar{A})^{\frac{k+1-P}{P}}(1+\alpha L_{\mathrm{max}})^{2P-2}<\frac{1}{2}$ holds, then the dynamic regret (\ref{DRR}) achieved by Algorithm \ref{alg5-Ch4} satisfies
	\begin{equation*}
	R_T^d\leq \frac{L}{2-4B_6}(2C_{T,2}+C_2),
	\end{equation*} 
	where $\bar{A}=\frac{1}{P}(\alpha-\frac{\alpha^2 L_{\mathrm{max}}}{2})$ and $C_2=|x_1-x_1^*|^2-2|x_1^*-x_0^*|^2$.
\end{thm}
\begin{pf}
	\color{black}Following similar steps \color{black} as those \color{black} in the proof of Theorem \ref{T7-Ch4} and using (\ref{Convbound-alg3}) and (\ref{Lybounds-alg2-v2}) yield
	\begin{equation}\label{Comparisonbound-alg3}
	\begin{split}
	&|x_{t+1}-x_t^*|\\
	&\leq {\frac{L}{\mu}}(1-2\mu \bar{A})^{\frac{k+1-P}{P}}(1+\alpha L_{\mathrm{max}})^{2P-2}|x_{t}-x_t^*|,
	\end{split}
	\end{equation}
	where $\bar{A}=\frac{1}{P}(\alpha-\frac{\alpha^2 L_{\mathrm{max}}}{2})$ as in the proof of Theorem \ref{T6-Ch4}. Since $\bar{A}>0$, we know there exists $k$ such that $B_6={\frac{L}{\mu}}(1-2\mu \bar{A})^{\frac{k+1-P}{P}}(1+\alpha L_{\mathrm{max}})^{2P-2}<\frac{1}{2}$. Hence, we have
	\begin{equation}\label{statebounds-alg3-v2}
	|x_{t+1}-x_t^*|^2\leq B_6|x_t-x_t^*|^2.
	\end{equation}
	From (\ref{statebounds-alg3-v2}), the following inequality holds by summation of both sides of (\ref{compa2_1-v2}) from $t=1$ to $t=T-1$
	\begin{equation*}
		\begin{split}
			(1-2B_6)2R_T^d/L&\leq |x_1-x_1^*|^2+2\sum_{t=1}^{T}|x_{t}^*-x_{t+1}^*|^2\\
			&=2C_{T,2}+C_2.
		\end{split}
	\end{equation*}
	Since $2B_6<1$, we have $R_T^d\leq \frac{L}{2-4B_6}(2C_{T,2}+C_2)$.\hfill $\qed$
\end{pf}

\section{Numerical simulations}\label{NS4}
\color{black}First, we study the following unconstrained quadratic problem
\begin{equation}\label{Simu-ch4}
\underset{x\in\mathbb{R}^{20}}{\text{min}}\ \frac{1}{2}x^TQ_tx-b^Tx,
\end{equation}
where $b\in\mathbb{R}^{20}$ is a randomly generated constant vector, $Q_t$ is a time-varying matrix that is positive definite for all $t\geq 1$. Moreover, all elements of $Q_t$ are in the closed interval $[1,2]$. As a result, \eqref{Simu-ch4} satisfies Assumptions \ref{a2-Ch4}-\ref{a3-CH4}. Next, we discuss Assumption \ref{a1-Ch4} (i) made in Theorem \ref{T4-Ch4}. The constant $G$ from Assumption \ref{a1-Ch4} (i) is only used to show $R_T^d\leq G\mathbb{E}[\sum_{t=1}^{T}|x_t-x_t^*|]$ such that \eqref{2ndtolast} holds. All other arguments made in the proof of Theorem \ref{T4-Ch4} are still true even if $G=\infty$. Thus, every iteration of Algorithm \ref{alg1-Ch4} will move $x_t$ closer to $x_t^*=Q_t^{-1}b$ expectation-wise, if the stepsize is small enough. Since in our setup $Q_t^{-1}b$ is uniformly bounded, the expectation of $x_t$ will remain bounded too. This means $x_t$ and the gradient $Q_tx_t-b$ must be bounded, almost surely\color{black}. We choose $P=20$ and for each $1\leq i\leq20$, $x_{(i)}$ is a scalar. The horizon length is $T=5000$, and the constant stepsize is chosen as $\alpha_t=\alpha=0.001$. Since the static regret is always upper bounded by the dynamic regret, we only show the plots for dynamic regrets and their time-averages $R_T^d/T$. 
\begin{figure}[ht]
	\includegraphics[width=8.4cm]{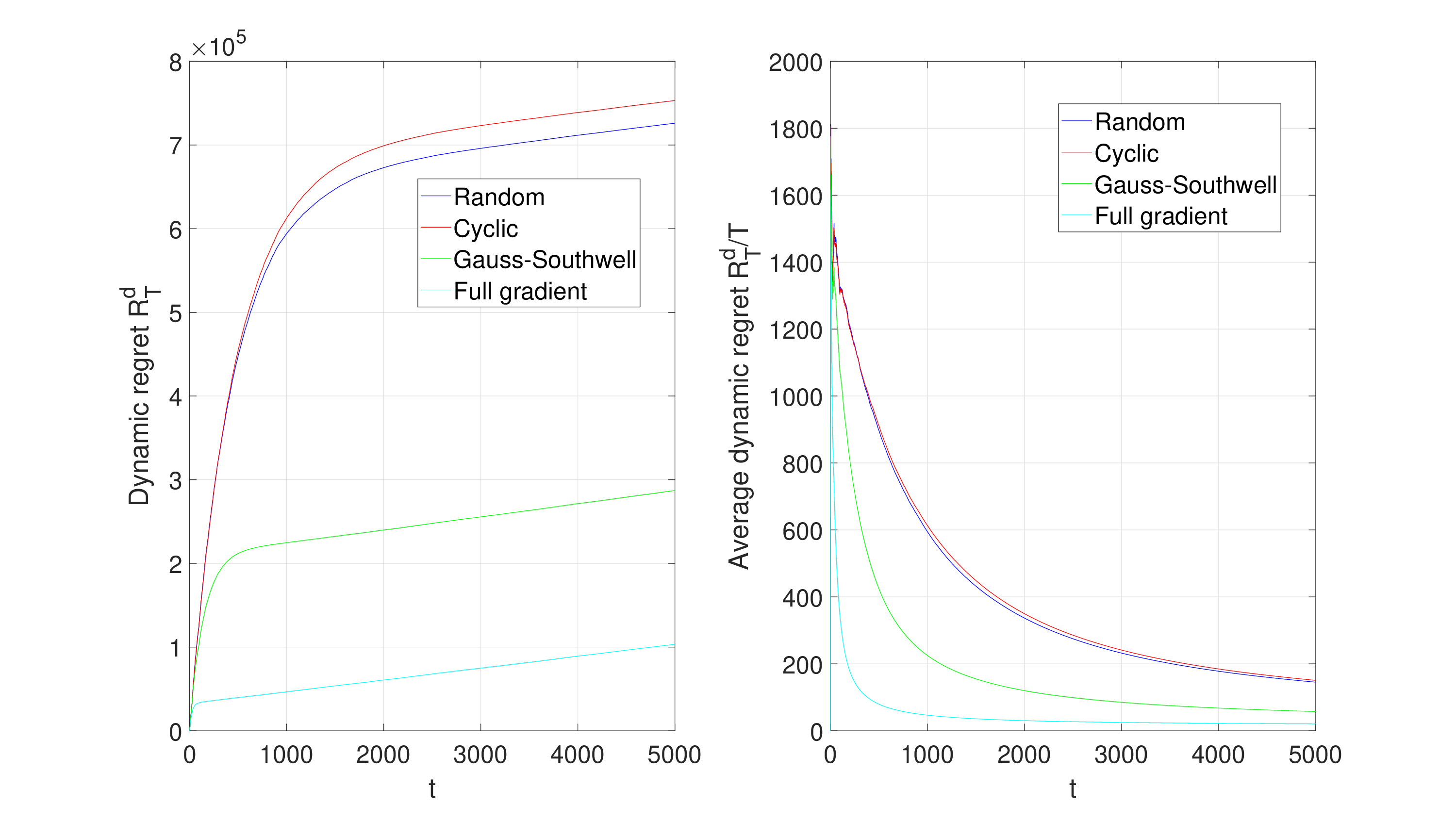}    
	\caption{Plots of the dynamic regrets $R_T^d$ and their time-averages $R_T^d/T$.}
	\label{revised_DR_Normal}
\end{figure}

It can be seen from Fig.~\ref{revised_DR_Normal} that when the constant stepsize is chosen sufficiently small, the regrets in all cases have sublinear growths and therefore their time-averages go to $0$ when $T$ is sufficiently large. This is consistent with our theoretical results from Theorem \ref{T4-Ch4}-\ref{T8-Ch4} on strongly convex functions. 
\begin{figure}[ht]
	\includegraphics[width=8.4cm]{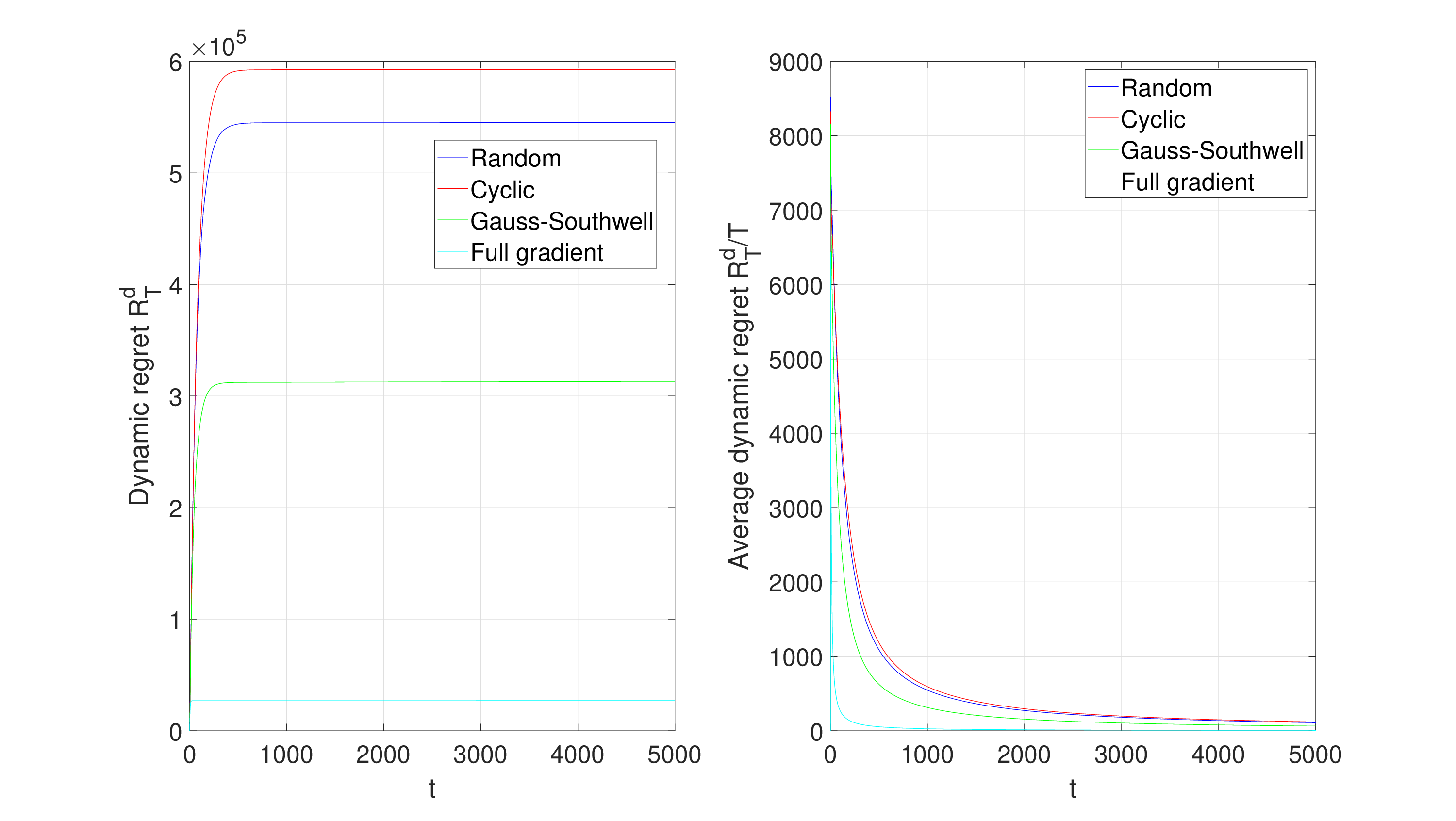}    
	\caption{Plots of the dynamic regrets $R_T^d$ and their time-averages $R_T^d/T$ with slow variation.}
	\label{revised_DR_slow}
\end{figure}
In order to see how the variation of the problem impacts the performance of the algorithms. We add an extra term of $100I_{20}$ such that the matrix $Q_t$ is diagonally dominant and therefore being less sensitive to $t$. We test the online algorithms in this case and the results are shown in Fig.~\ref{revised_DR_slow}. It can be seen that when problem (\ref{Simu-ch4}) varies slowly with respect to time, the curves of the regrets in Fig.~\ref{revised_DR_slow} have a lower \color{black} growth rate \color{black} compared to the regrets shown in Fig.~\ref{revised_DR_Normal}.

\color{black}As expected, the algorithm using full gradient has the best performance in terms of minimizing the dynamic regret. Yet, it is worth mentioning that among the three coordinate descent algorithms considered for this numerical example, Gauss-Southwell rule gives the best performance which is consistent with Remark \ref{Remark-compare}\color{black}. The extra information of the component-wise gradient norms enables a better selection of the coordinate to update. An in-depth theoretical analysis of this problem in an online setting is left for future work.

Next, we consider the following problem of minimizing entropy functions online
\begin{equation}\label{Simu-new}
	\underset{x\in\Theta}{\text{min}}  \sum_{i=1}^{5}\frac{x_{(i)}}{p_{i,t}}\ln\frac{x_{(i)}}{p_{i,t}}.
\end{equation}
The variable $x\in\Theta$ is decomposed into 5 scalar components with the compact constraint set \color{black} $\Theta=\{x\in\mathbb{R}^5\ |\ 0.001\leq x_{(i)}\leq1000, i=1,2,3,4,5\}$\color{black}. The values of $p_{i,t}$ are such that each $p_{i,1}$ is individually and randomly selected from $[1,5]$. For $t\geq2$, $p_{i,t}$ is such that $p_{i,t}=p_{i,t-1}+\frac{1}{t-1}$ for all $i$. It can be verified that the above selection ensures that $|x_{t+1}^*-x_t^*|=\frac{1}{t}$ and $C_T=O(\log T)$. Note that the cost function in \eqref{Simu-new} is convex but not strongly convex and hence only Theorem \ref{T1-Ch4} and Theorem \ref{T3-Ch4} apply. We again show the plots of dynamics regrets of Algorithms \ref{alg1-Ch4}-\ref{alg3-Ch4} and full gradient based algorithms in Fig.~\ref{entropy} with constant stepsize $\alpha=0.05$ and $T=5000$. Moreover, the plot for random coordinate descent is averaged over 1000 runs. 

It can be seen from the Fig.~\ref{revised_DR_Normal} to Fig.~\ref{entropy} that, for the quadratic problem, the dynamic regrets are a lot flatter when $t$ is large. On the other hand, the dynamic regrets in Fig.~\ref{entropy} still exhibit a significant growth when $t=T=5000$ even if we select the time-varying parameters to ensure that $C_T=O(\log T)$. These findings are consistent with the improved regret bounds shown in Theorems \ref{T4-Ch4}-\ref{T8-Ch4} for uniformly strongly convex functions with uniformly Lipschitz gradients.\color{black} 
\begin{figure}[ht]
	\includegraphics[width=8.4cm]{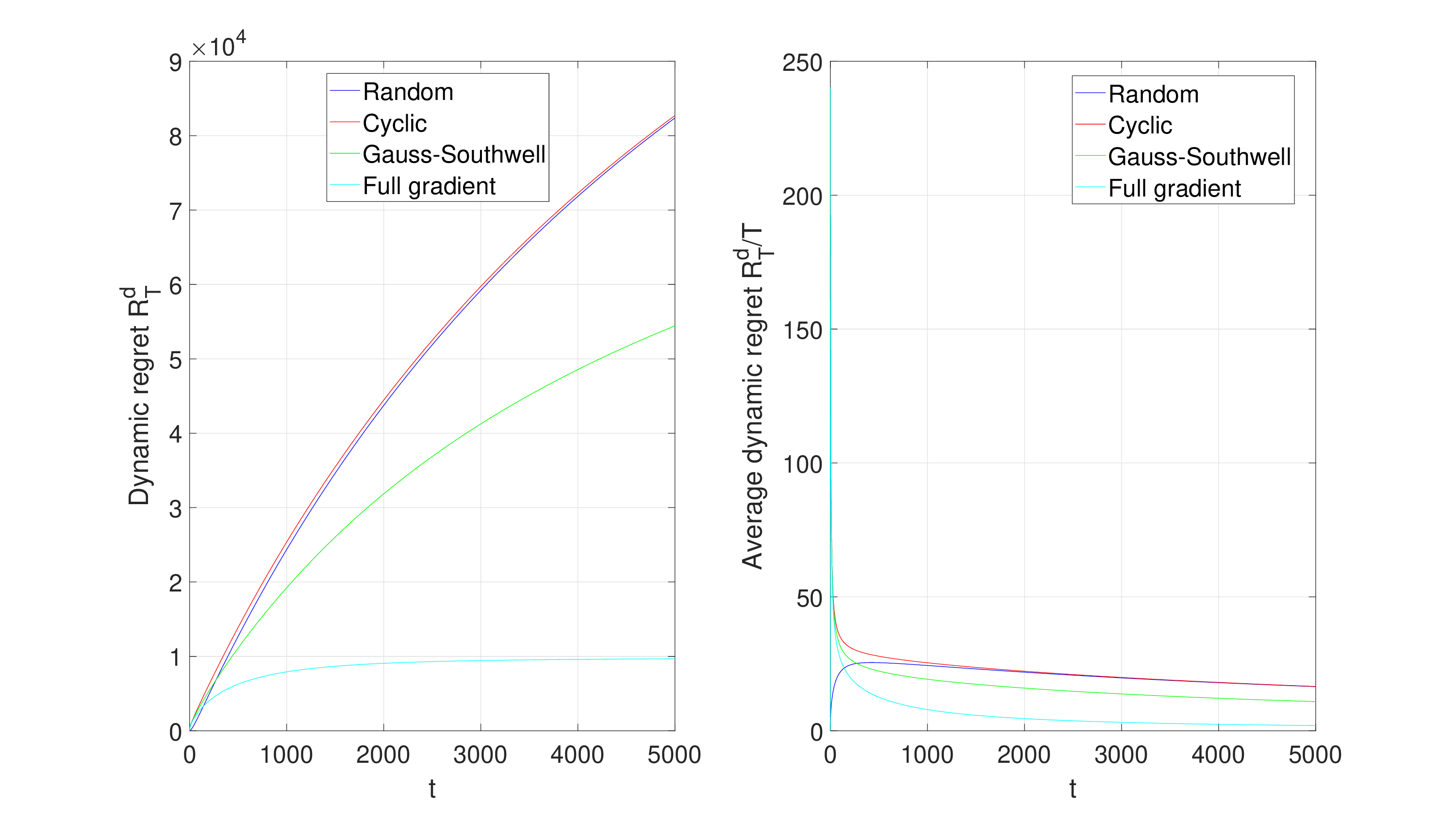}    
	\caption{Plots of the dynamic regrets $R_T^d$ and their time-averages $R_T^d/T$ in the non-strongly convex case.}
	\label{entropy}
\end{figure}
\section{Summary}\label{Summary4}
In this work, we have proposed an online coordinate descent algorithms to deal with optimization problems that may change over time. Three widely used update rules of coordinate descent are considered. Under different assumptions, we have provided different upper bounds on the regrets of these online algorithms. In particular, we have verified that the established regret bounds of these coordinate descent algorithms are of similar orders as those of online gradient descent methods under same settings. The regret bounds proved in this paper are summarized in Table \ref{table1-Ch4}. Lastly, a numerical example was given to illustrate our main result\color{black}. The possibilities of using coordinates with overlapping components is an interesting future research direction, especially for the deterministic case. Another topic of interest is the use of inaccurate gradient information in online coordinate descent algorithms.

\begin{ack}                               
	The authors would like to thank Dr. Yuen-Man Pun for pointing out an error in the proof of Theorem \ref{T7-Ch4} in the initial draft.  
\end{ack}
\bibliographystyle{plain}        
\bibliography{autosam}           

\begin{thebibliography}{10}

\bibitem{beck2013convergence}
Amir Beck and Luba Tetruashvili.
\newblock On the convergence of block coordinate descent type methods.
\newblock {\em SIAM journal on Optimization}, 23(4):2037--2060, 2013.

\bibitem{BSR18}
A.~S. Bedi, P.~Sarma, and K.~Rajawat.
\newblock Tracking moving agents via inexact online gradient descent algorithm.
\newblock {\em IEEE Journal of Selected Topics in Signal Processing},
  12(1):202--217, 2018.

\bibitem{BT}
D.~P. Bertsekas and J.~N. Tsitsiklis.
\newblock {\em Parallel and Distributed Computation: Numerical Methods}.
\newblock Prentice-Hall, 1989.

\bibitem{Bert1}
Dimitri~P Bertsekas.
\newblock {\em Convex Optimization Algorithms}.
\newblock Athena Scientific, 2015.

\bibitem{cao2022decentralized}
Xuanyu Cao and Tamer Ba{\c{s}}ar.
\newblock Decentralized online convex optimization with feedback delays.
\newblock {\em IEEE Transactions on Automatic Control}, 67(6):2889--2904, 2022.

\bibitem{cao2019online}
Xuanyu Cao and KJ~Ray Liu.
\newblock Online convex optimization with time-varying constraints and bandit
  feedback.
\newblock {\em IEEE Transactions on Automatic Control}, 64(7):2665--2680, 2019.

\bibitem{cao2021online}
Xuanyu Cao, Junshan Zhang, and H~Vincent Poor.
\newblock Online stochastic optimization with time-varying distributions.
\newblock {\em IEEE Transactions on Automatic Control}, 66(4):1840--1847, 2021.

\bibitem{chang2021online}
Ting-Jui Chang and Shahin Shahrampour.
\newblock On online optimization: Dynamic regret analysis of strongly convex
  and smooth problems.
\newblock In {\em Proceedings of the AAAI Conference on Artificial
  Intelligence}, volume~35, pages 6966--6973, 2021.

\bibitem{chen2017online}
Tianyi Chen, Qing Ling, and Georgios~B Giannakis.
\newblock An online convex optimization approach to proactive network resource
  allocation.
\newblock {\em IEEE Transactions on Signal Processing}, 65(24):6350--6364,
  2017.

\bibitem{garber2016linearly}
Dan Garber and Elad Hazan.
\newblock A linearly convergent variant of the conditional gradient algorithm
  under strong convexity, with applications to online and stochastic
  optimization.
\newblock {\em SIAM Journal on Optimization}, 26(3):1493--1528, 2016.

\bibitem{hazan2007logarithmic}
Elad Hazan, Amit Agarwal, and Satyen Kale.
\newblock Logarithmic regret algorithms for online convex optimization.
\newblock {\em Machine Learning}, 69(2-3):169--192, 2007.

\bibitem{hosseini2016online}
Saghar Hosseini, Airlie Chapman, and Mehran Mesbahi.
\newblock Online distributed convex optimization on dynamic networks.
\newblock {\em IEEE Transactions on Automatic Control}, 61(11):3545--3550,
  2016.

\bibitem{jenatton2016adaptive}
Rodolphe Jenatton, Jim Huang, and C{\'e}dric Archambeau.
\newblock Adaptive algorithms for online convex optimization with long-term
  constraints.
\newblock In {\em International Conference on Machine Learning}, pages
  402--411, 2016.

\bibitem{LT18}
A.~Lesage-Landry and J.~A. Taylor.
\newblock Setpoint tracking with partially observed loads.
\newblock {\em IEEE Transactions on Power Systems}, 33(5):5615--5627, 2018.

\bibitem{lesage2020predictive}
Antoine Lesage-Landry, Iman Shames, and Joshua~A Taylor.
\newblock Predictive online convex optimization.
\newblock {\em Automatica}, 113:108771, 2020.

\bibitem{lesage2020second}
Antoine Lesage-Landry, Joshua~A Taylor, and Iman Shames.
\newblock Second-order online nonconvex optimization.
\newblock {\em IEEE Transactions on Automatic Control}, 2020.

\bibitem{li2022survey}
Xiuxian Li, Lihua Xie, and Na~Li.
\newblock A survey of decentralized online learning.
\newblock {\em arXiv preprint arXiv:2205.00473}, 2022.

\bibitem{li2021online}
Yingying Li, Guannan Qu, and Na~Li.
\newblock Online optimization with predictions and switching costs: Fast
  algorithms and the fundamental limit.
\newblock {\em IEEE Transactions on Automatic Control}, 66(10):4761--4768,
  2021.

\bibitem{li2018distributed}
Zhenhong Li, Zhengtao Ding, Junyong Sun, and Zhongkui Li.
\newblock Distributed adaptive convex optimization on directed graphs via
  continuous-time algorithms.
\newblock {\em IEEE Transactions on Automatic Control}, 63(5):1434--1441, 2018.

\bibitem{LSN}
Yankai Lin, Iman Shames, and Dragan Ne{\v{s}}i{\'c}.
\newblock Asynchronous distributed optimization via dual decomposition and
  block coordinate ascent.
\newblock In {\em 58th IEEE Conference on Decision and Control}, pages
  6380--6385, 2019.

\bibitem{lin2021asynchronous}
Yankai Lin, Iman Shames, and Dragan Ne{\v{s}}i{\'c}.
\newblock Asynchronous distributed optimization via dual decomposition and
  block coordinate subgradient methods.
\newblock {\em IEEE Transactions On Control Of Network Systems},
  8(3):1348--1359, 2021.

\bibitem{mafakheri2022first}
Behnam Mafakheri, Iman Shames, and Jonathan Manton.
\newblock First order online optimisation using forward gradients under
  polyak-{L}ojasiewicz condition.
\newblock {\em arXiv preprint arXiv:2211.15825}, 2022.

\bibitem{mokhtari2016online}
Aryan Mokhtari, Shahin Shahrampour, Ali Jadbabaie, and Alejandro Ribeiro.
\newblock Online optimization in dynamic environments: {I}mproved regret rates
  for strongly convex problems.
\newblock In {\em 55th IEEE Conference on Decision and Control}, pages
  7195--7201, 2016.

\bibitem{nesterov2012efficiency}
Yu~Nesterov.
\newblock Efficiency of coordinate descent methods on huge-scale optimization
  problems.
\newblock {\em SIAM Journal on Optimization}, 22(2):341--362, 2012.

\bibitem{nutini2015coordinate}
Julie Nutini, Mark Schmidt, Issam Laradji, Michael Friedlander, and Hoyt
  Koepke.
\newblock Coordinate descent converges faster with the gauss-southwell rule
  than random selection.
\newblock In {\em International Conference on Machine Learning}, pages
  1632--1641. PMLR, 2015.

\bibitem{pang2023randomized}
Yipeng Pang and Guoqiang Hu.
\newblock Randomized gradient-free distributed online optimization via a
  dynamic regret analysis.
\newblock {\em IEEE Transactions on Automatic Control}, 2023.

\bibitem{paternain2017online}
Santiago Paternain and Alejandro Ribeiro.
\newblock Online learning of feasible strategies in unknown environments.
\newblock {\em IEEE Transactions on Automatic Control}, 62(6):2807--2822, 2017.

\bibitem{shahrampour2018distributed}
Shahin Shahrampour and Ali Jadbabaie.
\newblock Distributed online optimization in dynamic environments using mirror
  descent.
\newblock {\em IEEE Transactions on Automatic Control}, 63(3):714--725, 2018.

\bibitem{Shwartz12Found}
S.~Shalev-Shwartz.
\newblock Online learning and online convex optimization.
\newblock {\em Foundations and Trends in Machine Learning}, 4(2):107--194,
  2012.

\bibitem{shames2019online}
Iman Shames, Daniel Selvaratnam, and Jonathan~H Manton.
\newblock Online optimization using zeroth order oracles.
\newblock {\em IEEE Control Systems Letters}, 4(1):31--36, 2019.

\bibitem{Wright}
S.~J. Wright.
\newblock Coordinate descent algorithms.
\newblock {\em Mathematical Programming}, 151(1):3--34, 2015.

\bibitem{xiong2022distributed}
Yongyang Xiong, Xiang Li, Keyou You, and Ligang Wu.
\newblock Distributed online optimization in time-varying unbalanced networks
  without explicit subgradients.
\newblock {\em IEEE Transactions on Signal Processing}, 70:4047--4060, 2022.

\bibitem{yan2013distributed}
Feng Yan, Shreyas Sundaram, S.~V.~N Vishwanathan, and Yuan Qi.
\newblock Distributed autonomous online learning: {R}egrets and intrinsic
  privacy-preserving properties.
\newblock {\em IEEE Transactions on Knowledge and Data Engineering},
  25(11):2483--2493, 2013.

\bibitem{yi2015stochastic}
Peng Yi and Yiguang Hong.
\newblock Stochastic sub-gradient algorithm for distributed optimization with
  random sleep scheme.
\newblock {\em Control Theory and Technology}, 13:333--347, 2015.

\bibitem{yi2020distributedsp}
Xinlei Yi, Xiuxian Li, Lihua Xie, and Karl~H Johansson.
\newblock Distributed online convex optimization with time-varying coupled
  inequality constraints.
\newblock {\em IEEE Transactions on Signal Processing}, 68:731--746, 2020.

\bibitem{yi2020distributed}
Xinlei Yi, Xiuxian Li, Tao Yang, Lihua Xie, Tianyou Chai, and Karl~H Johansson.
\newblock Distributed bandit online convex optimization with time-varying
  coupled inequality constraints.
\newblock {\em IEEE Transactions on Automatic Control}, 66:4620--4635, 2021.

\bibitem{yi2023regret}
Xinlei Yi, Xiuxian Li, Tao Yang, Lihua Xie, Tianyou Chai, and Karl~H Johansson.
\newblock Regret and cumulative constraint violation analysis for distributed
  online constrained convex optimization.
\newblock {\em IEEE Transactions on Automatic Control}, 68:2875--2890, 2023.

\bibitem{yu2017online}
Hao Yu, Michael~J Neely, and Xiaohan Wei.
\newblock Online convex optimization with stochastic constraints.
\newblock In {\em Advances in Neural Information Processing Systems}, pages
  1428--1438, 2017.

\bibitem{yuan2022distributed}
Deming Yuan, Alexandre Proutiere, and Guodong Shi.
\newblock Distributed online optimization with long-term constraints.
\newblock {\em IEEE Transactions on Automatic Control}, 67(3):1089--1104, 2022.

\bibitem{zhang2017improved}
Lijun Zhang, Tianbao Yang, Jinfeng Yi, Rong Jin, and Zhi-Hua Zhou.
\newblock Improved dynamic regret for non-degenerate functions.
\newblock {\em Advances in Neural Information Processing Systems}, 30, 2017.

\bibitem{zinkevich2003online}
Martin Zinkevich.
\newblock Online convex programming and generalized infinitesimal gradient
  ascent.
\newblock In {\em Proceedings of the 20th international conference on machine
  learning}, pages 928--936, 2003.

\end{thebibliography}

\end{document}